\documentclass[12pt]{amsart}
\usepackage[top=3.8cm, right=3cm, left=3cm]{geometry}

\usepackage{texproject/styles/local-palatino}

\usepackage{texproject/macros/local-arxiv-typesetting}
\usepackage{texproject/macros/local-general}
\usepackage{texproject/macros/local-refs_ams}
\usepackage{texproject/macros/local-theorem_ams}
\usepackage{texproject/macros/local-analysis}
\usepackage{texproject/macros/local-tikz}
\usepackage{project-macros}
\newcommand{\titlename}{Local Dimensions of Self-similar Measures Satisfying the Finite Neighbour Condition}
\newcommand{\shorttitlename}{FNC Local Dimensions}
\newcommand{\docclasses}{28A80}
\newcommand{\dockeywords}{iterated function system, self-similar, local dimension, multifractal analysis, weak separation condition}

\subjclass[2020]{\docclasses}

\begin{document}
\title[\shorttitlename]{\titlename}
\author{Kathryn E. Hare \and Alex Rutar}
\address{Dept. of Pure Mathematics, University of Waterloo, Waterloo, Ontario N2L 3G1, Canada}
\email{kehare@uwaterloo.ca}
\address{Mathematical Institute, North Haugh, St Andrews, Fife KY16 9SS, Scotland}
\email{alex@rutar.org}
\thanks{KEH was supported by NSERC Grant 2016-03719.
AR was supported by this grant as well as EPSRC Grant EP/V520123/1}

\begin{abstract}
    We study sets of local dimensions for self-similar measures in $\mathbb{R}$ satisfying the finite neighbour condition, which is formally stronger than the weak separation condition but satisfied in all known examples.
    Under a mild technical assumption, we establish that the set of attainable local dimensions is a finite union of (possibly singleton) compact intervals.
    The number of intervals is bounded above by the number of non-trivial maximal strongly connected components of a finite directed graph construction depending only on the governing iterated function system.
    We also explain how our results allow computations of the sets of local dimensions in many explicit cases.
    This contextualizes and generalizes a vast amount of prior work on sets of local dimensions for self-similar measures satisfying the weak separation condition.
\end{abstract}

\maketitle
\tableofcontents
\section{Introduction}
A natural question when studying Borel probability measures on the real line, in particular those which are not absolutely continuous with respect to Lebesgue measure, is to quantify the singularity of the measure.
The Hausdorff dimension of the measure provides one coarse measurement.
A more fine-grained approach is through the local dimensions of the measure $\mu$ at points $x$ in its support, namely, the quantities
\begin{equation*}
    \dim_{\loc}\mu (x)=\lim_{r\to 0}\frac{\log \mu (B(x,r))}{\log r}.
\end{equation*}
In this paper, we are interested in determining properties of the set of attainable local dimensions for a given measure.

Our focus is on the invariant measures associated with an iterated function system (IFS) of similarities on $\mathbb{R}$, also known as self-similar measures.
These measures are simple to describe (see \cref{e:ifs} for the definition), yet exhibit rich and complex behaviour.
Historically, such measures have been of great interest.

Investigation of the sets of local dimensions of self-similar measures is related to multifractal analysis, in which one studies dimensional properties of the level sets of the local dimension function.
A heuristic relationship, known as the multifractal formalism \cite{hp1983}, implies (when it is satisfied) that the set of local dimensions is a closed interval.
The multifractal formalism holds if the IFS satisfies the classical open set condition (OSC) and there are simple formulas for the endpoints of the interval of attainable local dimensions \cite{cm1992,pat1997}.
But when the OSC fails to hold, the situation is much more complicated and less is known.

In \cite{hl2001}, Hu and Lau discovered that when $\mu$ is the 3-fold convolution of the classical middle-third Cantor measure, the set of local dimensions of $\mu$ consists of a closed interval along with an isolated point.
Generalizations of this example were studied in \cite{hhn2018,shm2005}, for example, while Testud \cite{tes2006a} gave an example of a Cantor-like measure, but with some of the similarities in the IFS having negative contraction factors, whose set of local dimensions was the union of two disjoint (non-trivial) intervals.
Another much studied family of self-similar measures which fail the OSC are the Bernoulli convolutions.
These are the measures associated with the IFS $\{\rho x,\rho x+1-\rho \}$ where $1/2<\rho <1$.
(See \cite{var2018} for more background on Bernoulli convolutions.)
It was shown by Feng \cite{fen2005} that when $\rho$ is the reciprocal of a simple Pisot number, such as the Golden mean, the set of local dimensions of the corresponding uniform Bernoulli convolution is, again, a closed interval.
However, all biased Bernoulli convolutions (regardless of the choice of $\rho $) and unbiased Bernoulli convolutions with contraction ratio greater than the reciprocal of the Golden mean have an isolated point in their set of local dimensions \cite{hh2019}.
We refer the reader to \cref{s:exifs} for more discussion on these important examples.

Convolutions of the middle-third Cantor measure and the Bernoulli convolutions with contraction factor the reciprocal of a Pisot number are all examples of self-similar measures associated with IFSs that satisfy the weak separation condition (WSC) \cite{ln1999}.
This separation condition is similar to the open set condition but allows exact overlaps \cite{zer1996}.
For such measures, the second author recently established the existence of a directed transition graph that encodes the local behaviour of the measure, and related the multifractal analysis of the measure with connectivity properties of the graph \cite{ruttoappear}.
One corollary of this earlier work is that when the transition graph is strongly connected, the set of attainable local dimensions of the measure is a closed interval.

In this paper, we significantly extend this local dimension result beyond the strongly connected case to obtain a more thorough understanding of sets of attainable local dimensions.
We specialize slightly to the case where the transition graph is finite, which we call the finite neighbour condition.
This separation condition is closely related to the generalized finite type condition defined by Lau and Ngai \cite{ln2007}.
The finite neighbour condition is equivalent to the weak separation condition when the support of the measure is an interval \cite{hhr2021}.
Our main contribution is to establish under the finite neighbour condition, and a weak technical assumption, that the set of local dimensions is a finite union of (possibly singleton) intervals.
Moreover, the number of intervals is bounded above by the number of non-trivial maximal strongly connected components of the transition graph.

Our research generalizes and contextualizes the prior analysis of sets of local dimensions for overlapping iterated function systems satisfying the weak separation condition.
We should emphasize that, in contrast with much of the earlier work on this problem, we do not require the IFS to have similarities with commensurable contraction factors.
Moreover, we are not aware of any examples of self-similar measures in $\mathbb{R}$, satisfying the weak separation condition, to which our results do not apply.

\subsection{Organization of the paper}

The main content of the paper is separated into two conceptual components: analysis of a graph-theoretic symbolic case, and specialization to self-similar measures.

First, in \cref{s:gd-mps}, we introduce a general weighted matrix product system.
This symbolic formalism can be thought of as a weighted generalization of the matrix-valued functions on shift space studied by past authors \cite{fen2003a,fen2009,fl2002}.
Under an irreducibility hypothesis similar to \cite{fen2009}, and using modified versions of the techniques contained therein, we establish in \cref{t:symb-lset} that the corresponding sets of Lyapunov exponents form a closed interval.
We also establish in \cref{l:per-dens} the density of Lyapunov exponents at special types of paths for which local dimension computations are particularly straightforward; this is useful in the computation of sets of local dimensions for specific examples.

In \cref{s:ifs-mps}, we review the details of the transition graph construction from \cite{ruttoappear} with a particular focus on self-similar measures $\mu$ on $\mathbb{R}$ that satisfy the finite neighbour condition (see \cref{d:fnc}).
This construction establishes the existence of a finite directed graph such that infinite paths in the graph correspond (almost) injectively to points in the support of $\mu$.
In fact, this directed graph construction is our motivation for studying the general matrix product systems.
The $\mu$-measure of a rich set of intervals (generating the topology on the support of $\mu$) is determined by products of non-negative matrices.
The weights in the matrix product system allow us to handle non-equicontractive IFS.

Then, in \cref{s:mf-properties}, we apply the results from the symbolic case to the study of the sets of local dimensions for these measures.
The relevant transition graph can be decomposed into finitely many non-trivial strongly connected components, which we refer to as maximal loop classes.
Any infinite path in the graph is eventually in exactly one maximal loop class, so maximal loop classes correspond to particular subsets of the support of $\mu$.
Under a technical assumption - that each maximal loop class satisfies either a simplicity or irreducibility hypothesis (see \cref{d:loop-assum}) - we relate the local dimensions at points corresponding to a maximal loop class to the Lyapunov exponents of the associated matrix product system.
This allows us to establish in \cref{c:loop-set} that the set of local dimensions at points corresponding to such a maximal loop class forms a closed interval.
Consequently, in \cref{c:attainable} we deduce that the set of attainable local dimensions of the measure is a finite union of intervals, some of which could be degenerate, with the number of intervals bounded above by the number of maximal loop classes.
The same results hold for upper local dimensions as well.

Lastly, in \cref{s:exifs}, we illustrate these ideas with examples, including those mentioned above.

\subsection{Some questions}
\begin{enumerate}[nl]
    \item We do not know if every self-similar measure in $\mathbb{R}$ that satisfies the weak separation condition also satisfies our formally stronger finite neighbour condition, or if every measure satisfying the finite neighbour condition satisfies the required technical assumption.
        If not, it would be of interest to extend the analysis.

    \item Our results establish that the sets of local dimensions and sets of upper local dimensions coincide.
        However, the set of lower local dimensions can be different, as seen in \cref{r:ldim-exception}.
        In that example, the set of lower local dimensions is still, however, a finite union of intervals corresponding to maximal loop classes.
        It is of interest to determine if similar results hold for sets of lower local dimensions.
\end{enumerate}

\subsection{Notation}
The reals $\R$ are a metric space with the usual Euclidean metric, and $\N$ is the set of natural numbers beginning at 1.
The set $B(x,r)$ is a closed ball centred at $x$ with radius $r$.
Given a set $E\subseteq\R$, we write $\diam(E)=\sup\{|x-y|:x,y\in E\}$.

Given a set $X$, we write $\# X$ to denote the cardinality of $X$.
Given two real-valued functions $f(z), g(z)$ defined on some index set $Z$, we write $f\succcurlyeq g$ (resp. $f\preccurlyeq g$) if there exists some $c>0$ such that $f(z)\geq cg(z)$ (resp. $f(z)\leq cg(z)$) for each $z\in Z$.
We say $f\asymp g$ if $f\succcurlyeq g$ and $f\preccurlyeq g$.

If $M$ is a square matrix, we denote by $\spr M$ the spectral radius of $M$.
All matrices in this document are non-negative.

\subsection{Acknowledgements}
The authors would like to thank K. G. Hare for many helpful conversations.

\section{Graph-directed matrix product systems}\label{s:gd-mps}
\subsection{Basic definitions}\label{ss:wpm}
Let $G$ be a finite directed graph with vertex set $V(G)$ and edge set $E(G)$.
We will assume that $G$ is \defn{strongly connected}, which means that there is a directed path connecting any two vertices.
Each vertex $v$ has a \defn{dimension} $d(v)\in\N$, and to each edge $e=(v_1,v_2)$ we associate a non-negative $d(v_1)\times d(v_2)$ \defn{transition matrix} $T(e)$ and a \defn{weight} $W(e)\in(0,1)$.
We let
\begin{equation*}
    d_{\max}=\max_{v \in G} d(v).
\end{equation*}

Let $\Sigma$ denote the set of all infinite paths $(e_i)_{i=1}^\infty$ in $G$ and let $\Sigma^*$ denote the set of all finite paths in $G$.
A path is a \defn{cycle} if it begins and ends at the same vertex.
The \defn{length} of a finite path is the number of edges it contains.
We say a path $\eta\in\Sigma^*$ is a \defn{prefix} of a (finite or infinite) path $\gamma$ if $\gamma=\eta \gamma'$ for some path $\gamma'$.
Given $\gamma=(e_n)_{n=1}^\infty\in\Sigma$, we write $\gamma|n=(e_1,\ldots,e_n)\in\Sigma^*$ to denote the unique prefix of length $n$.

Given $\eta=(e_1,\ldots,e_n)\in\Sigma^*$, we write
\begin{equation*}
    W(\eta)=W(e_1)\cdots W(e_n)
\end{equation*}
and if $\eta\in\Sigma^*$ has length at least 1, $\eta^-=(e_1,\ldots,e_{n-1})$.
For convenience, let
\begin{align*}
    W_{\min}&=\min\{W(e):e\in E(G)\}>0, & W_{\max}&=\max\{W(e):e\in E(G)\}<1.
\end{align*}
If $\eta$ is the empty path, we say $W(\eta)=1$.
Similarly, we write
\begin{equation*}
    T(\eta)=T(e_1)\cdots T(e_n)\text{ for }\eta=(e_1,\ldots,e_n)\in\Sigma^*.
\end{equation*}
We equip $\Sigma$ with the topology induced by the metric
\begin{equation*}
    d(\gamma,\xi)=\inf\{W(\eta):\eta\text{ a prefix of }\gamma\text{ and }\xi\}.
\end{equation*}
With this topology, $\Sigma$ is a compact totally disconnected metric space.

We refer to this data as a \defn{graph-directed matrix product system} or, in short, a \defn{matrix product system}.
Typically, we will denote this by $\mathcal{G}$.

\begin{definition}\label{LyExp}
    Given an infinite path $\gamma=(e_j)_{j=1}^\infty\in\Sigma$, we define the \defn{lower Lyapunov exponent} by
    \begin{equation*}
        \underline{\lambda}(\mathcal{G},\gamma) = \liminf_{n\to\infty}\frac{\log\norm{T(e_1)\cdots T(e_n)}}{\log W(e_1)\cdots W(e_n)}.
    \end{equation*}
    The \defn{upper Lyapunov exponent} is defined similarly; when the values coincide, we call this value the \defn{Lyapunov exponent} of the path $\gamma$, and denote it by $\lambda(\mathcal{G},\gamma)$.
    Typically, we omit writing $\mathcal{G}$ when it is clear from the context.
\end{definition}
For any $t>0$, denote
\begin{equation*}
    \Sigma_t = \{\eta\in \Sigma^*:W(\eta)< t\leq W(\eta^-),\norm{T(\eta)}>0\},
\end{equation*}
which is the set of paths with non-zero transition matrix and weight approximately $t$.

\subsection{Irreducible matrix product systems}
It is clear that the geometric properties of the metric space $\Sigma$ are determined completely from the edge weights.
However, in order to say meaningful things about products of matrices and Lyapunov exponents, we require a stronger form of irreducibility than the graph $G$ being strongly connected.
\begin{definition}
    We say that the matrix product system is \defn{irreducible} if there exists a finite family of paths $\mathcal{H}\subset\Sigma^*$ such that for any vertices $v_1,v_2$, $1\leq i\leq d(v_1)$, and $1\leq j\leq d(v_2)$, there exists a path $\gamma\in\mathcal{H}$ from vertex $v_1$ to $v_2$ such that $T(\gamma)_{i,j}>0$.
\end{definition}
\begin{remark}
    Equivalently, for each $1\leq i,j\leq m=\# V(G)$, define $M_{i,j}=T(e)$ if there is an edge $e$ from vertex $v_i$ to $v_j$, and let $M_{i,j}=0$ otherwise.
    The matrix product system is irreducible if and only if the block matrix
    \begin{equation*}
        M=
        \begin{pmatrix}
            M_{1,1}&\cdots &M_{1,m}\\
            \vdots&\ddots&\vdots\\
            M_{m,1}&\cdots &M_{m,m}
        \end{pmatrix}
    \end{equation*}
    is irreducible, i.e. there exists some $r>0$ such that $\sum_{k=1}^r M^k$ is a strictly positive matrix.

    Of course, irreducible systems are necessarily strongly connected.
\end{remark}
\begin{remark}
    Our irreducibility criterion is very similar to the one assumed by Feng \cite{fen2009}.
    However, since our weights depend on the edge rather than the source vertex, we find it more natural to speak of infinite paths in a graph rather than words in a sequence space.
    One may equivalently think of the graph as a subshift of finite type determined by a weighted adjacency matrix.
\end{remark}
For the remainder of this section, unless otherwise stated, our matrix product system is irreducible.

Irreducibility is essential for obtaining the following estimates, which we will use frequently.
\begin{lemma}\label{l:irred}
    There are constants $A,B>0$ such that for any paths $\eta_1,\eta_2\in\Sigma^*$, there exists some $\gamma\in\mathcal{H}$ such that $\eta_1\gamma\eta_2$ is a path and
    \begin{equation*}
        A\norm{T(\eta_1)}\norm{T(\eta_2)}\leq\norm{T(\eta_1\gamma\eta_2)}\leq B\norm{T(\eta_1)}\norm{T(\eta_2)}
    \end{equation*}
\end{lemma}
\begin{proof}
    By the irreducibility assumption, for any $v_1,v_2\in V(G)$, $1\leq i\leq d(v_1)$, and $1\leq j\leq d(v_2)$, there exists a path $\gamma=\gamma(v_1,v_2,i,j)\in\mathcal{H}$ from $v_1$ to $v_2$ such that $T(\gamma)_{i,j}>0$.
    Let
    \begin{align*}
        C &= \min\{T(\gamma(v_1,v_2,i,j))_{i,j}:v_1,v_2\in V(G),1\leq i\leq d(v_1),1\leq j\leq d(v_2)\}.
    \end{align*}

    If $\eta_1,\eta_2$ are arbitrary paths, by the pigeonhole principle, there exists some $k,i,j,\ell$ such that $d_{\max}^2T(\eta_1)_{k,i}\geq \norm{T(\eta_1)}$ and $d_{\max}^2T(\eta_2)_{j,\ell}\geq \norm{T(\eta_2)}$.
    Assume $\eta_1$ ends at vertex $v_1$, $\eta_2$ begins at vertex $v_2$, and take $\gamma=\gamma(v_1,v_2,i,j)\in\mathcal{H}$.
    Then $\eta_1\gamma\eta_2$ is a path and
    \begin{equation*}
        \norm{T(\eta_1\gamma\eta_2)}\geq T(\eta_1)_{k,i}T(\gamma)_{i,j}T(\eta_2)_{j,\ell}\geq\frac{C}{d_{\max}^4}\norm{T(\eta_1)}\norm{T(\eta_2)}.
    \end{equation*}
    The lower bound follows by taking $A=C/d_{\max}^4$.

    To obtain the upper bound, we simply note that
    \begin{equation*}
        \norm{T(\eta_1\gamma\eta_2)}\leq\norm{T(\eta_1)}\norm{T(\gamma)}\norm{T(\eta_2)}
    \end{equation*}
    and it suffices to take $B=\max\{\norm{T(\gamma)}:\gamma\in\mathcal{H}\}$.
\end{proof}
In the following lemma, we do not formally need the irreducibility hypothesis: it suffices to know that if $\eta$ is any path in $\mathcal{G}$, then $T(\eta)$ is not the zero matrix.
\begin{lemma}\label{l:lapprox}
    There are constants $A,r>0$ such that for any $t_1,t_2\in(0,1)$, $\eta_1\in\Sigma_{t_1}$, and $\eta_2\in\Sigma_{t_2}$, there are paths $\phi$ and $\psi$ such that $\eta_1\phi\eta_2\psi\in\Sigma_{t_1t_2 r}$ and
    \begin{equation*}
        \norm{T(\eta_1\phi\eta_2\psi)}\leq A\norm{T(\eta_1)}\norm{T(\eta_2)}.
    \end{equation*}
\end{lemma}
\begin{proof}
    Take $r=\min\{W(\eta):\eta\in\mathcal{H}\}$.
    By the irreducibility hypothesis, there exists some $\phi\in\mathcal{H}$ such that $\norm{T(\eta_1\phi\eta_2)}>0$.

    Moreover, for any path $\eta$ with $\norm{T(\eta)}>0$, there exists an edge $e$ such that $\eta e$ is a path and $\norm{T(\eta e)}>0$.
    Since $t_1t_2W_{\min}^{-2}\geq W(\eta_1\phi\eta_2)\geq t_1t_2 r$, repeatedly applying this observation, there exists $\psi$ such that $\eta_1\phi\eta_2\psi\in\Sigma_{t_1t_2 r}$.
    Note that $W(\psi)\geq rW_{\min}^2$.
    Thus $\norm{T(\eta_1\phi\eta_2\psi)}\leq A\norm{T(\eta_1)}\norm{T(\eta_2)}$ where
    \begin{equation*}
        A = \max\{\norm{T(\eta)}:\eta\in E(\mathcal{G}),W(\eta)\geq rW_{\min}^2\}\cdot \max\{\norm{T(\eta):\eta\in\mathcal{H}}\}
    \end{equation*}
    as required.
\end{proof}
\begin{lemma}\label{l:irredp}
    There are constants $A,B>0$ such that for any path $\eta\in\Sigma^*$, there exists some $\phi\in\mathcal{H}$ such that $\eta\phi$ is a cycle and
    \begin{equation*}
        A\norm{T(\eta)}\leq\spr T(\eta\phi)\leq B\norm{T(\eta)}.
    \end{equation*}
\end{lemma}
\begin{proof}
    Let $C$ be the minimal strictly positive coefficient of any $T(\phi)$ for $\phi\in\mathcal{H}$.
    Suppose $T(\eta)_{i,\ell}$ is the maximal coordinate of $T(\eta)$.
    Get $\phi\in\mathcal{H}$ such that $\eta\phi$ is a cycle and $T(\phi)_{\ell,i}\geq C>0$.
    Then
    \begin{align*}
        \Tr T(\eta\phi)  &\geq T(\eta\phi)_{i,i}\geq T(\eta)_{i,\ell}T(\phi)_{\ell,i}\geq\norm{T(\eta)}\frac{C}{d_{\max}^2}.
    \end{align*}

    Since the trace of a matrix is the sum of its eigenvalues,
    \begin{equation*}
        \Tr(T(\eta\phi))\leq d_{\max}\spr T(\eta\phi).
    \end{equation*}
    Thus with $A=C/d_{\max}^3$, we have $\spr(T(\eta\phi))\geq A\norm{T(\eta)}$.

    Conversely, we have
    \begin{equation*}
        \spr(T(\eta\phi))\leq\norm{T(\eta\phi)}\leq\norm{T(\eta)}\norm{T(\phi)}\leq B\norm{T(\eta)}
    \end{equation*}
    where $B=\max\{\norm{T(\gamma)}:\gamma\in\mathcal{H}\}$.
\end{proof}

\subsection{Attainable Lyapunov exponents}
The main goal of this subsection is to determine the possible values of Lyapunov exponents of paths in the matrix product system.

We begin with notation.
Put
\begin{align}\label{e:a-min-max}
    \alpha_{\min}(\mathcal{G}) =  \alpha_{\min}&:= \lim_{t\to 0}\min_{\eta\in\Sigma_t}\frac{\log \norm{T(\eta)}}{\log t},&\alpha_{\max}(\mathcal{G})= \alpha_{\max}&:= \lim_{t\to 0}\max_{\eta\in\Sigma_t}\frac{\log\norm{T(\eta)}}{\log t}.
\end{align}
We will first show that $\alpha_{\min}$ and $\alpha_{\max}$ are well defined and take real values.
This will use the following standard submultiplicativity result, which is a slightly modified version of, for example, \cite[Theorem~7.6.1]{hp1957}.
\begin{lemma}\label{l:su-m}
    Let $f\colon(0,1)\to\R$ be measurable and suppose there exists $c>0$ and $r>0$ such that $f(t_1t_2 r)\leq c\cdot f(t_1)f(t_2)$ for all $t_1,t_2\in(0,1)$.
    Then 
    \begin{equation*}
        \lim_{t\to 0}\frac{\log f(t)}{\log t} = \inf_{t>0}\frac{\log f(t)}{\log t}.
    \end{equation*}
\end{lemma}
Note the similarity of the following lemma with \cite[Lemma~2.3]{fen2009}.
\begin{lemma}\label{l:a-lim}
    The limits defining $\alpha_{\min}$ and $\alpha_{\max}$ exist and take real values.
\end{lemma}
\begin{proof}
    We will first prove that the limit
    \begin{equation*}
        \lim_{t\to 0}\min_{\eta\in\Sigma_t}\frac{\log \norm{T(\eta)}}{\log t} = \lim_{t\to 0}\frac{\log \max_{\eta\in\Sigma_t}\norm{T(\eta)}}{\log t}
    \end{equation*}
    exists.
    Set $f(t)=\max_{\eta\in\Sigma_t}\norm{T(\eta)}$.
    Let $t_1,t_2>0$ be arbitrary.
    If $\eta\in\Sigma_{t_1t_2 W_{\min}}$, we may write $\eta=\eta_1\eta_2\gamma$ where $\eta_1\in\Sigma_{t_1}$ and $\eta_2\in\Sigma_{t_2}$ and $W(\gamma)\geq W_{\min}^2$.
    In particular, with $c=\max\{\norm{T(\psi)}:W(\psi)\geq W_{\min}^2\}$, we have $\norm{T(\eta)}\leq c\norm{T(\eta_1)}\norm{T(\eta_2)}$ and therefore $f(t_1t_2 W_{\min})\leq c\cdot f(t_1)f(t_2)$.
    Applying \cref{l:su-m} with $c$ as above and $r=W_{\min}$, we have our desired result.

    We now show that $\lim_{t\to 0}\max_{\eta\in\Sigma_t}\frac{\log\norm{T(\eta)}}{\log t}$ exists.
    Set $g(t)=\min_{\eta\in\Sigma_t}\norm{T(\eta)}$.
    Let $t_1,t_2>0$ and let $\eta_1\in\Sigma_{t_1}$ and $\eta_2\in\Sigma_{t_2}$ be arbitrary.
    Note that $\eta_1\eta_2$ need not be a path, and even if it were, it need not hold that $\eta_1\eta_2\in\Sigma_{t_1t_2}$.
    By \cref{l:lapprox}, there exists some $c,r>0$ (not depending on $\eta_1$ and $\eta_2$) and paths $\phi$ and $\psi$ such that $\eta_1\phi\eta_2\psi$ is an admissible path in $\Sigma_{rt_1t_2}$ and
    \begin{equation}\label{e:cp1}
        g(rt_1t_2)\leq c\norm{T(\eta_1)}\norm{T(\eta_2)}.
    \end{equation}
    Now taking the minimum over all $\eta_1\in\Sigma_{t_1}$ and $\eta_2\in\Sigma_{t_2}$ yields $g(rt_1t_2)\leq cg(t_1)g(t_2)$.
    Thus $g$ satisfies \cref{l:su-m}.

    To see that $\alpha_{\min},\alpha_{\max}\in\R$, let $a$ be the smallest strictly positive entry in any $T(e)$ for $e\in E(G)$.
    Let $b=\min\{\norm{T(e)}:e\in E(G)\}$.
    Then if $\eta\in\Sigma_t$ is any path of length $n$, we have that
    \begin{equation*}
        \frac{\log b}{\log W_{\min}} \leq \frac{\log \norm{T(\eta)}}{\log t}\leq\frac{n\log a}{(n-1)\log W_{\max}}
    \end{equation*}
    so that $\alpha_{\min},\alpha_{\max}$ are real-valued.
\end{proof}
Of course, if $\eta\in\Sigma_t$, then $W(\eta)\asymp t$.
Consequently,
\begin{equation*}
    \alpha_{\min}=\lim_{t\to 0}\min_{\eta\in\Sigma_t}\frac{\log\norm{T(\eta)}}{\log W(\eta)} \text{ and }  \alpha_{\max}=\lim_{t\to 0}\max_{\eta\in\Sigma_t}\frac{\log\norm{T(\eta)}}{\log W(\eta)}.
\end{equation*}

We are now ready to prove the following result about the set of attainable Lyapunov exponents.
We remind the reader that the Lyapunov exponent of the path $\gamma$, $\lambda(\gamma)$, was defined in \cref{LyExp}.

Our proof follows \cite[Lemma~2.3~and~Proposition~3.2]{fen2009}.
\begin{theorem}\label{t:symb-lset}
    Let $\mathcal{G}$ be a matrix product system satisfying the irreducibility hypothesis.
    \begin{enumerate}[nl,r]
        \item If $\gamma\in\Sigma$ is any path, then $\underline{\lambda}(\gamma),\overline{\lambda}(\gamma)\in [\alpha_{\min},\alpha_{\max}]$.
        \item For any $\alpha\in[\alpha_{\min},\alpha_{\max}]$ and $\xi\in\Sigma^*$ with $T(\xi)$ non-zero, there exists some $\gamma=(e_m)_{m=1}^\infty\in\Sigma$ and a sequence $(m_j)_{j=1}^\infty$ such that $\lambda(\gamma)=\alpha$, $\lim_{j\to\infty} m_{j+1}/m_j=1$, and for each $j\in\N$, $\xi$ is a prefix of $(e_{m_j},e_{m_j+1},\ldots)$.
    \end{enumerate}
\end{theorem}
\begin{proof}
    For (i), if $\gamma\in\Sigma$ is arbitrary, then
    \begin{equation*}
        \underline{\lambda}(\gamma)=\lim_{k\to\infty}\frac{\log\norm{T(\gamma|{n_k})}}{W(\gamma|{n_k})}
    \end{equation*}
    for some subsequence $n_k$.
    But if $\gamma|n_k\in \Sigma_{t_{n_k}}$, then $W(\gamma|{n_k})\asymp t_{n_k}$ and
    \begin{equation*}
        \alpha_{\min}\leq\lim_{k\to\infty}\frac{\log\norm{T(\gamma|{n_k})}}{\log t_{n_k}}\leq\alpha_{\max}
    \end{equation*}
    from the existence of the limits defining $\alpha_{\min}$ and $\alpha_{\max}$.
    The upper Lyapunov exponent result is identical, giving (i).

    Now for (ii), given $\alpha \in [\alpha_{\min},\alpha_{\max}]$, let $s\in[0,1]$ be such that $\alpha=s\alpha_{\min}+(1-s)\alpha_{\max}$.
    For each $n\in\N$, choose $\phi_n, \psi_n\in\Sigma_{2^{-n}}$ with the property that
    \begin{equation*}
        \frac{\log\norm{T(\phi_n)}}{\log W(\phi_n)}=u_n \rightarrow \alpha_{\min}  \text{ and } \frac{\log\norm{T(\psi_n)}}{\log W(\psi_n)}=v_n \rightarrow \alpha_{\max}.
    \end{equation*}
    Let $\{A_n\}_{n=1}^\infty$, $\{B_n\}_{n=1}^\infty$ be sequences of natural numbers given by
    \begin{align*}
        A_n &= [sn]\\
        B_n &= [(1-s)n],
    \end{align*}
    where $[x]$ denotes the integer part of $x$.
    Then define a sequence by
    \begin{equation*}
        \underbrace{\phi_1,\ldots,\phi_1}_{A_1},\underbrace{\psi_1,\ldots,\psi_1}_{B_1},\ldots,\underbrace{\phi_n,\ldots,\phi_n}_{A_n},\underbrace{\psi_n,\ldots,\psi_n}_{B_n},\ldots
    \end{equation*}
    and relabel it $\{\eta_n\}_{n=1}^\infty$, i.e. $\eta_1=\phi_1$, $\eta_{A_1}=\phi_1$, $\eta_{A_1+1}=\psi_1$, etc.

    Now since $\xi\in\Sigma^*$ has $T(\xi)$ non-zero, by repeatedly applying \cref{l:irred}, there are constants $C_1,C_2>0$ such that for each $n\in\N$ there are paths $\nu_{n}^{(1)},\nu_n^{(2)}$ in $\mathcal{H}$ such that with $\nu_n:=\nu_n^{(1)}\xi\nu_{n}^{(2)}$,
    \begin{equation*}
        \gamma\coloneqq(\eta_1,\nu_1,\eta_2,\nu_2,\ldots)
    \end{equation*}
    is an infinite path and
    \begin{equation}\label{e:npr}
        C_1^m\prod_{i=1}^m\norm{T(\eta_i)}\leq \norm{T(\eta_1\nu_1\ldots\eta_m\nu_m)}\leq C_2^m\prod_{i=1}^m\norm{T(\eta_i)}.
    \end{equation}
    Since $\mathcal{H}$ is a finite set, there also exists $D_1,D_2>0$ such that
    \begin{equation}\label{e:dpr}
        D_1^m\prod_{i=1}^m W(\eta_i)\leq W(\eta_1\nu_1\ldots\eta_m\nu_m)\leq D_2^m\prod_{i=1}^m W(\eta_i).
    \end{equation}
    For notation, let $\{n_k\}_{k=1}^\infty$ be the indices such that $n_k$ is the index of the edge preceding the first edge of $\phi_{k+1}$ in repetition $A_{k+1}$.

    Let $(m_j)_{j=1}^\infty$ be the sequence of indices such that $\xi$ is a prefix and fix some $j\in\N$.
    For any $\ell\in\N$, since $\phi_\ell$ and $\psi_\ell$ are in $\Sigma_{2^{-\ell}}$, there exists some $a,b>0$ such that $a\ell\leq |\phi_\ell|\leq b\ell$ where $|\phi_\ell|$ is the number of edges in $\phi_\ell$.
    Moreover, $(\gamma_{m_j},\gamma_{m_j+1},\ldots)$ has prefix $\xi\zeta_1\phi_\ell\zeta_2\xi$ or $\xi\zeta_1\psi_\ell\zeta_2\xi$ where $\ell$ is chosen suitably and $\zeta_1,\zeta_2\in\mathcal{H}$ have bounded length.
    Thus there exists some $M>0$ such that $m_{j+1}-m_j\leq M+b \ell$.
    On the other hand, it always holds that $m_j\geq\sum_{i=1}^{\ell-1} a i$.
    It follows that $\lim_{j\to\infty}m_{j+1}/m_j=1$, as claimed

    We now prove that $\overline{\lambda}(\gamma)\leq\alpha$; the lower bound $\underline{\lambda}(\gamma)\geq\alpha$ will follow by a similar argument.
    To this end, let $n$ be a large number of edges and let $k$ be maximal such that $n\geq n_k$ (that is, $k$ is the maximal number of completed blocks $A_i,B_i$ which occur before edge $n$).
    There exist constants $C_3,C_4>0$ such that
    \begin{equation*}
        C_3^{n_{k+1}-n_k}\norm{T(\gamma_1 \ldots\gamma_{n_{k+1}})} \leq \norm{T(\gamma | n)} \leq C_4^{n_{k+1}-n_k}\norm{T(\gamma_1 \ldots \gamma_{n_k})}
    \end{equation*}
    and
    \begin{equation*}
        W(\gamma_1 \ldots \gamma_{n_{k+1}}) \leq W(\gamma | n) \leq W(\gamma_1 \ldots \gamma_{n_k}).
    \end{equation*}

    Since the number of edges contained in $\gamma |n$ is at least $\sum_{i=1}^k (A_i+B_i)$ and at most $\sum_{i=1}^{k+1}(A_i+B_i)$, we deduce from \cref{e:npr} and \cref{e:dpr} that
    \begin{align*}
        &\frac{\log \norm{T(\gamma|n)}}{\log W(\gamma|n)}\\
        &\qquad\leq \frac{(n_{k+1}-n_k)\log C_3+\sum_{i=1}^{k+1} \bigl((A_i+B_i)\log C_1+ A_i\log \norm{T(\phi_i)}+B_i\log \norm{T(\psi_i)}\bigr)}{\sum_{i=1}^{k} \bigl((A_i+B_i)\log D_2+A_i\log W(\phi_i)+B_i\log W(\psi_i)\bigr)}.
    \end{align*}

    Since each $\phi_i,\psi_i\in\Sigma_{t_{2^{-i}}}$, we have $W(\phi_i)\asymp W(\psi_i)\asymp 2^{-i}$.
    Recall, also, that  $A_i, B_i  \asymp i$.
    Therefore
    \begin{equation*}
        \Bigl\lvert\frac{\sum_{i=1}^{k+1} (A_i+B_i)\log C_1}{\sum_{i=1}^k (A_i\log W(\phi_i)+B_i\log W(\psi_i))}\Bigr\rvert\preccurlyeq\frac{\sum_{i=1}^{k+1}i}{\sum_{i=1}^k i^2}\preccurlyeq\frac{1}{k} \rightarrow 0
    \end{equation*}
    and a similar statement holds with the numerator replaced by $\sum_{i=1}^k(A_i+B_i)\log D_2$.
    Moreover, since $(n_{k+1}-n_k)\log C_3 \asymp (k+1)(A_{k+1}+B_{k+1})$, we also have
    \begin{equation*}
        \Bigl\lvert\frac{(n_{k+1}-n_k)\log C_3}{\sum_{i=1}^k (A_i\log W(\phi_i)+B_i\log W(\psi_i))}\Bigr\rvert\rightarrow 0
    \end{equation*}
    We thus have that
    \begin{align*}
        \limsup_n \frac{\log \norm{T(\gamma|n)}}{\log W(\gamma|n)}&\leq  \limsup_k  \frac{\sum_{i=1}^{k+1} (A_i \log \norm{T(\phi_i)}+B_i\log \norm{T(\psi_i)})}{\sum_{i=1}^{k} (A_i\log W(\phi_i)+B_i\log W(\psi_i))}\\
                                                                                              &=\limsup_k  \frac{\sum_{i=1}^{k+1} (A_i u_i \log W(\phi_i)+B_i v_i\log W(\psi_i))}{\sum_{i=1}^{k} (A_i\log W(\phi_i)+B_i\log W(\psi_i))} \\
                                                                                              &=\limsup_k  \frac{\sum_{i=1}^{k+1} (iA_i u_i +i B_i v_i)}{\sum_{i=1}^{k} (iA_i+i B_i)}\\
                                                                                              &=\limsup_k  \frac{\sum_{i=1}^{k+1} i^2(s u_i + (1-s)v_i)}{\sum_{i=1}^{k} i^2}.
    \end{align*}

    Fix $\epsilon >0$.
    Since $\lim_{i\to\infty}(su_i+(1-s)v_i)=\alpha$, for large enough $N$, $su_i +(1-s)v_i \leq \alpha +\epsilon$ for all $i \geq N$.
    Thus
    \begin{equation*}
        \limsup_n \frac{\log \norm{T(\gamma|n)}}{\log W(\gamma|n)} \leq \limsup_k  \frac{\sum_{i=N}^{k+1} i^2 (\alpha + \epsilon)}{\sum_{i=1}^{k} i^2} \leq \alpha + \epsilon.
    \end{equation*}

    Similar reasoning shows that 
    \begin{equation*}
        \liminf_n \frac{\log \norm{T(\gamma|n)}}{\log W(\gamma|n)} \geq \alpha -\epsilon.
    \end{equation*}
    As $\epsilon>0$ was arbitrary, it follows that $\lambda (\gamma) =\alpha$, as claimed.
\end{proof}
The following result now follows directly from \cref{t:symb-lset}.
\begin{corollary}\label{c:ly-set}
    Let $\mathcal{G}$ be an irreducible matrix product system.
    Then the set of attainable Lyapunov exponents is the compact interval $[\alpha_{\min},\alpha_{\max}]$.
\end{corollary}
\subsection{Density of periodic paths}\label{ss:d-per}
An interesting class of paths are the so-called \defn{periodic paths}, which are the paths in $\Sigma$ of the form
\begin{equation*}
    \gamma=(\theta,\theta,\ldots)
\end{equation*}
where $\theta$ is a cycle.
We denote them by $\mathcal{P}$.
We refer to $\theta$ as a \defn{period} of the path.

The Lyapunov exponent of a periodic path always exists and has a simple formula.
\begin{proposition}\label{p:per-ldim}
    Let $\gamma$ be a periodic path with period $\theta$.
    Then the Lyapunov exponent of $\gamma$ exists and is given by
    \begin{equation*}
        \lambda(\gamma)=\frac{\log \spr T(\theta)}{\log W(\theta)}.
    \end{equation*}
\end{proposition}
\begin{proof}
    Assume that $\theta = (\theta_1,...,\theta_k)$.
    For any positive integer $k$ and $j=1,\ldots,k$, 
    \begin{equation*}
        \norm{T(\theta^n\theta_1\cdots \theta_j)} \leq \norm{T(\theta^n)}\norm{T(\theta_1\cdots \theta_j)}
    \end{equation*}
    and
    \begin{equation*}
        \norm{T(\theta^{n+1})} \leq \norm{T(\theta^n\theta_1\cdots \theta_j)}\norm{T(\theta_{j+1}\cdots \theta_k)}.
    \end{equation*}
    Consequently, there is some $A,B>0$, depending only on $\theta$, such that
    \begin{equation*}
        A\norm{T(\theta^{n+1})} \leq\norm{T(\theta^n\theta_1\cdots \theta_j)}\leq B\norm{T(\theta^n)}.
    \end{equation*}
    The result follows directly from the fact that 
    \begin{equation*}
        \lim_{n\to\infty}\frac{\log\norm{T(\theta^n)}}{n}=\log \spr(T(\theta)).
    \end{equation*}
\end{proof}

\begin{proposition}\label{l:per-dens}
    The set $\{\lambda(\gamma):\gamma\in\mathcal{P}\}$ is dense in $[\alpha_{\min},\alpha_{\max}]$.
\end{proposition}
\begin{proof}
    It suffices to show that if $\gamma\in\Sigma$ is an arbitrary path such that $\lambda(\gamma)$ exists, there exists a sequence of periodic paths $\{y_n\}_{n=1}^\infty$ such that $\lim_{n\to\infty}\lambda(y_n)=\lambda(\gamma)$.

    By \cref{l:irredp}, there are constants $A,B>0$ such that for any $k\in\N$, there is a path $\eta_k\in\mathcal{H}$ such that $(\gamma|k)\eta_k\eqqcolon\theta_k$ is a cycle and
    \begin{equation*}
        A\norm{T(\gamma|k)}\leq\spr T(\theta_k)\leq B\norm{T(\gamma|k)}.
    \end{equation*}
    Let $\overline{\theta}_k=(\theta_k,\theta_k,\ldots)$.
    This is a periodic path with period $\theta_k$, so that
    \begin{equation*}
        \lambda(\overline{\theta_k})=\frac{\log\spr(T(\theta_k))}{\log W(\theta_k)}
    \end{equation*}
    by \cref{p:per-ldim}.
    Also, $W(\gamma|k) \asymp W(\theta_k)$.
    Hence
    \begin{align*}
        \limsup_{k\to\infty}\lambda(\overline{\theta_k}) &\leq \limsup_{k\to\infty}\frac{\log A\norm{T(\gamma|k)}}{\log W(\gamma|k)} = \lambda(\gamma)
    \end{align*}
    and the lower bound follows identically.
    Thus $\lim_{k\to\infty}\lambda(\overline{\theta_k})=\lambda(\gamma)$ and we have density, as claimed.
\end{proof}

\section{Iterated function systems and their matrix product systems}\label{s:ifs-mps}
We now turn to studying iterated function systems of similarities.
In this section, we will describe how the dynamics of associated self-similar sets and measures can be encoded with a matrix product system.

\subsection{The transition graph and the finite neighbour condition}\label{ss:tg-fnc}
We begin with notation and terminology.
By an iterated function system (IFS), $\{S_{i}\}_{i=1}^{m},$ we mean a finite set of similarities 
\begin{equation}  \label{e:ifs}
    S_{i}(x)=r_{i}x+d_{i}\colon\mathbb{R}\rightarrow \mathbb{R}\text{ for each } i=1,...,m
\end{equation}
with $0<\left\vert r_{i}\right\vert <1$ and $m\geq 1$.
We say that the IFS is \defn{equicontractive} if $r_1=\cdots=r_m>0$.

Each IFS generates a unique non-empty, compact set $K$ satisfying 
\begin{equation*}
    K=\bigcup_{i=1}^{m}S_{i}(K), 
\end{equation*}
known as the associated \defn{self-similar set}.
We will assume $K$ is not a singleton.
By translating the $d_{i}$ as necessary, without loss of generality we may assume that the convex hull of $K$ is $[0,1]$.

Given probabilities $\mathbf{p}=(p_{i})_{i=1}^{m}$ where $p_{i}>0$ and $\sum_{i=1}^{m}p_{i}=1$, there exists a unique Borel probability measure $\mus$ satisfying 
\begin{equation}\label{e:minv}
    \mus(E)=\sum_{i=1}^mp_{i}\mus(S_{i}^{-1}(E))
\end{equation}
for any Borel set $E\subseteq K$.
This non-atomic measure $\mus$ is known as an associated self-similar measure and has as its support the self-similar set $K$.

Given $\sigma =(\sigma _{1},\ldots ,\sigma _{j})\in \{1,...,m\}^{j}$, we
denote 
\begin{equation*}
    \sigma ^{-}=(\sigma _{1},\ldots ,\sigma _{j-1})\text{, }S_{\sigma}=S_{\sigma _{1}}\circ \cdots \circ S_{\sigma _{j}}\text{ and }r_{\sigma}=r_{\sigma _{1}}\cdots r_{\sigma _{j}}.
\end{equation*}
For $t>0,$ put 
\begin{equation*}
    \Lambda _{t}=\{\sigma :|r_{\sigma }|<t\leq |r_{\sigma ^{-}}|\}.
\end{equation*}
The elements of $\Lambda _{t}$ are called the \defn{words of generation $t$}.
We remark that in the literature it is more common to see this defined by the rule $|r_{\sigma }|\leq t<|r_{\sigma ^{-}}|$, but this essentially equivalent choice is more convenient for our purposes.

The notions of net intervals and neighbour sets were first introduced in \cite{fen2003} and extended in \cite{hhr2021,ruttoappear}.
We summarize the key ideas here.

Let $h_{1},\ldots ,h_{s(t)}$ be the collection of distinct elements of the set $\{S_{\sigma }(0),S_{\sigma }(1):\sigma \in \Lambda _{t}\}$ listed in strictly ascending order and set 
\begin{equation*}
    \mathcal{F}_{t}=\{[h_{j},h_{j+1}]:1\leq j<s(t)\text{ and }(h_{j},h_{j+1})\cap K\neq \emptyset \}.
\end{equation*}
The elements of $\mathcal{F}_{t}$ are called the \defn{net intervals of generation} $t$.
Note that $[0,1]$ is the (unique) net interval of any generation $t>1$ and denote by 
\begin{equation*}
    \mathcal{F}=\bigcup_{t>0}\mathcal{F}_{t} 
\end{equation*}
the set of all net intervals.

Given a net interval $\Delta$, we denote by $T_{\Delta}$ the unique similarity $T_{\Delta}(x)=rx+a$ with $r>0$ such that 
\begin{equation*}
    T_{\Delta }([0,1])=\Delta.
\end{equation*}
Of course, here $r=\diam(\Delta)$ and $a$ is the left endpoint of $\Delta$.

\begin{definition}\label{d:nb}
    We will say that a similarity $f(x)=Rx+a$ is a \defn{neighbour} of $\Delta \in \mathcal{F}_{t}$ if there exists some $\sigma \in \Lambda_{t}$ such that $S_{\sigma }(K)\cap\Delta^\circ\neq\emptyset $ and $f=T_{\Delta }^{-1}\circ S_{\sigma }$.
    In this case, we also say that $S_{\sigma }$ \defn{generates} the neighbour $f$.

    The \defn{neighbour set} of $\Delta $ is the maximal set 
    \begin{equation*}
        \vs_t(\Delta )=\{f_{1},\ldots ,f_{k}\}
    \end{equation*}
    where each $f_{i}=T_{\Delta }^{-1}\circ S_{\sigma_{i}}$ is a distinct neighbour of $\Delta$.
    We denote by
    \begin{equation*}
        \rmax(\Delta)\coloneqq\max\{|R|:x\mapsto Rx+a\in\vs_t(\Delta)\}
    \end{equation*}
    the maximum contraction factor of any neighbour of $\Delta$.
\end{definition}

When the generation is implicit, we will often write $\mathcal{V}(\Delta )$.
Since $K=\bigcup_{\sigma \in \Lambda_{t}}S_{\sigma }(K)$, every net interval $\Delta$ has a non-empty neighbour set.
\begin{remark}
    As explained in \cite[Remark~2.2]{ruttoappear}, for an equicontractive IFS $\{\lambda x+d_i\}_{i\in\mathcal{I}}$ with $0<\lambda<1$, our notion of neighbour set is closely related to Feng's neighbour and characteristic vector construction \cite{fen2003}.
    Instead of normalizing by some global factor of the form $\lambda^{n}$, we normalize locally with respect to $\diam(\Delta)$.
    In this case, the words of generation $\lambda^{n-1}$ are the words of length $n$ and the net intervals of generation $n$ (in Feng's notation) have diameter comparable to $\lambda^n$. 

    This is important since, outside the equicontractive case, there is no uniform notion of an integer-valued generation.
\end{remark}

We now discuss some illustrative examples of this construction.

Assume $\Delta \in \mathcal{F}_{t}$ has neighbour set $\{f_{1},\ldots,f_{k}\}$ and for each $i$, let $S_{\sigma _{i}}$ generate the neighbour $f_{i}$.
The \defn{transition generation} of $\Delta $, denoted $\tg(\Delta )$, is given by 
\begin{equation*}
    \tg(\Delta )=\max \{|r_{\sigma _{i}}|:1\leq i\leq k\}.
\end{equation*}
It is straightforward to verify that $\tg(\Delta )=\rmax(\Delta )\diam(\Delta )$.
The \defn{children} of (\defn{parent}) $\Delta$ are the net intervals of generation $\tg(\Delta )$ contained in $\Delta$.
We remark that if there is only one child, $\Delta_{1}$, then $\vs(\Delta )\neq \vs(\Delta_{1})$.
Given $\Delta =[a,b],$ with child $\Delta _{1}=[a_{1},b_{1}]$, we define the \textit{position index} $q(\Delta ,\Delta_{1})=(a_{1}-a)/\diam\Delta$.
The position index will enable us to distinguish children with the same neighbour set.

The children of a net interval are locally determined by the neighbour set of the net interval in the following sense.
\begin{theorem}[\cite{ruttoappear}, Theorem~2.8]\label{t:ttype}
    Let $\{S_i\}_{i=1}^m$ be an arbitrary IFS.
    Then for any $\Delta\in\mathcal{F}_t$ with children $(\Delta_1,\ldots,\Delta_n)$ in $\mathcal{F}_{\tg(\Delta)}$, the index $n$, neighbour sets $\vs(\Delta_i)$, position indices $q(\Delta,\Delta_i)$, and ratios $\tg(\Delta_i)/\tg(\Delta)$ depend only on $\vs(\Delta)$.
\end{theorem}
Thus much of the important information about the IFS is captured in the behaviour of the neighbour sets.
This motivates the construction of the directed \textit{transition graph}, $\mathcal{G}(\{S_{i}\}_{i=1}^{m})$, defined as follows.
The vertex set of $\mathcal{G}$, denoted $V(\mathcal{G})$, is the set of distinct neighbour sets, $V(\mathcal{G})=\{\vs(\Delta ):\Delta \in \mathcal{F}\}$.
For each parent/child pair of net intervals, $\Delta \in \mathcal{F}_{t}$ and $\Delta_{i}\in \mathcal{F}_{\tg(\Delta )}$, we introduce an edge $e=(\vs_{t}(\Delta),\vs_{\tg(\Delta)}(\Delta_{i}),q(\Delta ,\Delta _{i}))$.
Here $\vs_{t}(\Delta )$ is the source vertex and $\vs_{\tg(\Delta )}(\Delta _{i})$ is the target vertex.
We write $E(\mathcal{G})$ for the set of all edges.
By \cref{t:ttype}, this construction is well-defined since it depends only on the neighbour set of $\Delta $.

An \defn{(admissible) path} $\eta $ in $\mathcal{G}$ is a sequence of edges $\eta =(e_{1},\ldots ,e_{n})$ in $\mathcal{G}$ where the target of $e_{i}$ is the source of $e_{i+1}$.
A path in $\mathcal{G}$ is a \defn{cycle} if it begins and ends at the same vertex.
We denote by $\Sigma_0$ the set of infinite paths beginning at the root vertex $\vs([0,1])$, and $\Sigma_0^*$ the set of finite paths beginning at $\vs([0,1])$.

Nested sequences of net intervals are in correspondence with finite paths in $\Sigma_0^*$.
Given $\Delta \in \mathcal{F}_{t}$, consider the sequence $(\Delta _{0},\ldots ,\Delta _{n})$ where $\Delta _{0}=[0,1]$, $\Delta_{n}=\Delta $, and each $\Delta _{i}$ is a child of $\Delta _{i-1}$.
By the \defn{symbolic representation} of $\Delta$, we mean the finite path $\eta =(e_{1},\ldots,e_{n})$ in $\mathcal{G}$ where 
\begin{equation*}
    e_{i}=\bigl(\vs(\Delta _{i-1}),\vs(\Delta _{i}),q(\Delta _{i-1},\Delta _{i})\bigr)\text{ for each }i=1,...,n.
\end{equation*}
Conversely, if $\eta =(e_{1},\ldots ,e_{n})$ is any finite path, we say that $\eta$ is \defn{realized} by $(\Delta _{i})_{i=0}^{n}$ if each $\Delta_{i}$ is a child of $\Delta _{i-1}$ and each $e_{i}=(\vs(\Delta _{i-1}),\vs(\Delta _{i}),q(\Delta _{i-1},\Delta _{i}))$.
We denote the symbolic representation of $\Delta$ by $[\Delta]$.

\begin{definition}\label{d:syr}
    Given some $x\in K$, we say that an infinite path $\gamma\in\Sigma_0$ is a \defn{symbolic representation} of $x$ if
    \begin{equation*}
        \{x\}=\bigcap_{i=1}^\infty\Delta_i
    \end{equation*}
    where for each $n$, $[\Delta_n]$ is the symbolic representation of the length $n$ prefix of $\gamma$, denoted by $\gamma|n$.
    We say that $x$ is an \defn{interior point} of $K$ if $x$ has a unique symbolic representation.
\end{definition}
If $x$ is not an interior point, then $x$ must be an endpoint of two distinct net intervals at any sufficiently small scale.

\begin{definition}\label{d:e-len}
    Let $\mathcal{G}$ be the transition graph of an IFS.
    We define the \defn{edge weight}, $W\colon E(\mathcal{G})\to(0,1)$ by the rule that if edge $e$ has source $\vs(\Delta_1)$ and target $\vs(\Delta_2)$, then $W(e)=\tg(\Delta_2)/\tg(\Delta_1)$.
\end{definition}
This function is well-defined by \cref{t:ttype}.
We extend $W$ to finite paths by putting $W(\eta )=W(e_{1})\cdots W(e_{n})$ when $\eta =(e_{1},\ldots ,e_{n})$ .

An important observation is that if $\Delta \in \mathcal{F}_{t}$ is any net interval with symbolic representation $\eta $, then $W(\eta )\asymp t$, with constants of comparability not depending on $\Delta$.
While the above choice of the weight for an edge is not unique with this property, a straightforward argument shows that any such function must agree with $W$ on any cycle.

\begin{definition}\label{d:fnc}
    We say that the IFS $\{S_i\}_{i=1}^m$ satisfies the \defn{finite neighbour condition} if its transition graph is a finite graph.
\end{definition}
Equivalently, there are only finitely many neighbours.
We also say that the associated self-similar measure satisfies the finite neighbour condition, even though this condition does not depend on the choice of probabilities.

The finite neighbour condition was introduced in \cite{hhr2021} and explored in more detail in \cite[Section~5]{ruttoappear}.
In \cite{hhr2021} it was shown that the finite neighbour condition is equivalent to the generalized finite type condition holding with respect to the invariant open set $(0,1)$ (see \cite{ln2007} for the original definition of GFT) and hence satisfies the weak separation condition \cite{ln2007}.
In particular, all IFS that satisfy the open set condition or the finite type condition with respect to $(0,1)$ (see \cite{fen2003} for the definition of finite type) satisfy the finite neighbour condition.
For simplicity, throughout the remainder of this document, whenever we say that an IFS satisfies the finite type condition, we always mean with respect to $(0,1)$.

This includes examples such as the iterated function systems
\begin{equation*}
    \{\rho x,\rho x+(1-\rho )\}
\end{equation*}
where $\rho$ is the reciprocal of a Pisot number.
Here, the associated self-similar measures are the much studied Bernoulli convolutions (c.f., \cite{fen2005}, \cite{var2018} and the many references cited therein), or the overlapping Cantor-like IFS
\begin{equation*}
    \{x/d+i(d-1)/md:i=0,1,...,m\}
\end{equation*}
where $d\geq 3$ is a natural number (see \cite{hhn2018,shm2005}).
For example, in the case of the Bernoulli convolution with $\rho$ the reciprocal of the Golden mean, there are six neighbour sets.
These are listed in \cref{ex:bconv} and the transition graph is given in \cref{f:gm-graph}.

A non-equicontractive example is given by the IFS $\{\rho x,rx+\rho (1-r),rx+1-r\}$ where $\rho ,r>0$ satisfy $\rho +2r-\rho r\leq 1$.
This was introduced in \cite{lw2004} where it was shown to satisfy the WSC.
In fact, this IFS satisfies the finite neighbour condition (see \cite{ln2007} or \cite[Section~5.3]{ruttoappear}).
Note that it does not satisfy the open set condition (due to the existence of exact overlaps) and does not necessarily have commensurable contraction factors, so it cannot be of finite type.
See \cref{f:tgraph} for its transition graph and \cref{ex:lw} for more details about its structure.
Other examples of IFS satisfying the finite neighbour condition can also be found in \cref{s:exifs}.

In \cite[Theorem~4.4]{hhr2021} it was proven, under the assumption that the self-similar set is an interval, that the finite neighbour condition is equivalent to the weak separation condition.
It is unknown if the two properties coincide for IFS in $\mathbb{R}$.
Further details on these various separation conditions for IFS can be found in \cite{hhr2021}.

\subsection{Transition matrices}\label{ss:tg-im}
We now show how one can encode the measure of net intervals through the so-called transition matrices.

For the remainder of the paper, we fix a total order on the set of all neighbours $\{f:f\in \vs(\Delta ),\Delta \in \mathcal{F}\}$.
Let $e\in E(\mathcal{G})$ be an edge, say $e=(\vs(\Delta_{1}),\vs(\Delta _{2}),q(\Delta _{1},\Delta _{2}))$.
Assume the neighbour sets are given by $\vs(\Delta _{1})=\{f_{1},\ldots ,f_{k}\}$ and $\vs(\Delta_{2})=\{g_{1},\ldots ,g_{n}\}$ where $f_{1}<\cdots <f_{k}$ and $g_{1}<\cdots <g_{n}$.
We define the \defn{transition matrix} $T(e)$ as the non-negative $k\times n$ matrix given by 
\begin{equation}\label{e:trmat}
    T(e)_{i,j}=p_{\ell }
\end{equation}
if there exists an index $\ell \in \{1,...,m\}$ such that $f_{i}$ is generated by $\sigma $ and $g_{j}$ is generated by $\sigma \ell $; otherwise, set $T(e)_{i,j}=0$.
Note that this definition is slightly different than the original definition; see \cite[Section~5.2]{ruttoappear} for more detail concerning this.

It is clear from \cref{t:ttype} that this definition depends only on the edge $e$.
If $\eta =(e_{1},\ldots ,e_{n})$ is a path, we define 
\begin{equation*}
    T(\eta )=T(e_{1})\cdots T(e_{n}).
\end{equation*}
We refer to these matrices as transition matrices, as well.

Recall that if $\sigma ^{\prime}$ generates any neighbour of $\Delta_{2}$, then necessarily $\sigma ^{\prime }=\sigma \ell$ for some $\sigma $ which generates a neighbour of $\Delta _{1}$; thus, every column of $T(e)$ has a positive entry.
More generally, if $\eta$ is a path, then $T(\eta)$ has a positive entry in every column.
However, it may not hold that each row of $T(\eta)$ has a positive entry.

We continue to use the notation $\norm{T}=\sum_{i,j}T_{ij}$ for the matrix $1$-norm of a non-negative matrix $T$.

The following relationship between the measure of net intervals and transition matrices is known.
\begin{proposition}[\cite{ruttoappear}, Corollary~5.5]\label{c:dt-form}
    Suppose $\Delta$ is a net interval with symbolic representation $\eta$.
    Then
    \begin{equation*}
        \mus(\Delta) \asymp \norm{T(\eta)},
    \end{equation*}
    with constants of comparability not depending on the choice of $\Delta$.
\end{proposition}
Thus the transition matrices encode the distribution of $\mus$ on net intervals.

We conclude this subsection by mentioning the following straightforward property of transition matrices.
\begin{lemma}\label{l:left-prod}
    Let $\gamma$ be a finite path.
    Fix $n\in\N$ and let $\gamma'=(\gamma_{n+j})_{j=0}^\infty$.
    We have $\norm{T(\gamma')}\asymp\norm{T(\gamma)}$ with constant of comparability depending only on $n$.
\end{lemma}
\begin{proof}
    Write $\gamma=\eta\gamma'$ where $\eta$ is a path of length $n$.
    Since every transition matrix has a non-zero entry in each column, a straightforward calculation shows that there exists some constant $a=a(\eta)$ such that $\norm{T(\eta\gamma')}\geq a(\eta)\norm{T(\gamma')}$.
    On the other hand, $\norm{T(\eta\gamma')}\leq\norm{T(\eta)}\norm{T(\gamma')}$.
    But there are only finitely many paths $\eta$ of length $n$, giving the result.
\end{proof}

\subsection{Maximal loop classes and irreducibility}
From this point on we will assume that $\mathcal{G}$ is the matrix product system corresponding to an IFS that satisfies the finite neighbour condition.

Let $\mathcal{L}$ be an induced subgraph of $\mathcal{G}$ (i.e. $\mathcal{L}$ is the graph consisting of the vertices $V(\mathcal{L})$ and any edge $e\in E(\mathcal{G})$ such that $e$ connects two vertices in $\mathcal{L}$).
Of course, the induced subgraph naturally inherits a matrix product system from the full graph.
\begin{definition}
    We say that the subgraph $\mathcal{L}$ is a \defn{loop class} if for any vertices $v_1,v_2\in V(\mathcal{L})$, there is a non-empty directed path connecting $v_1$ and $v_2$, and we call $\mathcal{L}$ \defn{maximal} if it is maximal with this property.
\end{definition}

Two maximal loop classes necessarily have disjoint vertex sets, but not all vertices need to belong to a maximal loop class.
However, given any symbolic representation $\gamma=(\gamma_i)_{i=0}^{\infty}$, there is a unique maximal loop class $\mathcal{L}$ in which $\gamma$ is \defn{eventually}, meaning  there exists some $N$ such that $\gamma'\coloneqq(\gamma_i)_{i=N}^\infty$ is an element of $\Sigma(\mathcal{L})$.

We will let 
\begin{equation*}
    K_{\mathcal{L}} = \{x\in K: x \text{ has a symbolic representation that is eventually in }\mathcal{L}\}.
\end{equation*}
Every element of $K$ belongs to at least one set $K_{\mathcal{L}}$ for a maximal loop class $\mathcal{L}$, and at most two such sets.

Abusing notation slightly, given $\gamma$ which is eventually in $\mathcal{L}$, we write
\begin{equation*}
    \lambda(\mathcal{L},\gamma)=\lim_{n\to\infty}\frac{\log \norm{T(\gamma|n)}}{\log W(\gamma|n)}.
\end{equation*}
By \cref{l:left-prod}, for $k\geq N$ we have $\norm{T(\gamma|k)}\asymp\norm{T(\gamma'|k-N)}$ where $\gamma'$ is as above.
Since, also, $W(\gamma|k)\asymp W(\gamma'|k-N)$,  we have $\lambda(\mathcal{L},\gamma)=\lambda(\mathcal{L},\gamma')$ (and similarly for upper and lower Lyapunov exponents), where $\gamma'$ is a path in $\mathcal{L}$, justifying our notation.

We are primarily interested in three types of maximal loop classes.
\begin{definition}\label{d:loop-assum}
    \begin{enumerate}[nl,r]
        \item We say that a maximal loop class is \defn{irreducible} if the corresponding matrix product system is irreducible.
        \item We say that a maximal loop class is \defn{simple} if all cycles share the same edge set.
        \item We say that a maximal loop class is an \defn{essential class} if any vertex reachable from the maximal loop class by a directed path is also in the maximal loop class.
    \end{enumerate}
\end{definition}

For example, the IFS $\{\rho x,\rho x+1-\rho\}$ where $\rho$ is the reciprocal of the Golden ratio has an essential class with three elements, and two other singleton maximal loop classes.
These loop classes are all irreducible; for more details, see \cref{ex:bconv}.
Other examples are also given in \cref{s:exifs}.

Note that irreducibility is a statement about the IFS and does not depend on the choice of (non-zero) probabilities.

Any IFS satisfying the weak separation condition (such as those satisfying the finite neighbour condition) has a unique essential class by \cite[Proposition~3.3]{ruttoappear}.
In fact, the finite neighbour condition can be characterized by the property that the associated transition graph has a finite essential class \cite[Theorem~5.3]{ruttoappear}.

Moreover, the essential class is always irreducible; this is essentially shown in \cite[Lemma~3.9]{ruttoappear} (or \cite[Lemma~6.4]{fen2009} in the equicontractive case), but we include a self-contained proof here as an illustrative example:
\begin{proposition}\label{p:ess-irred}
    Let $\mathcal{G}$ be the transition graph of an IFS satisfying the finite neighbour condition with essential class $\mathcal{G}_{\ess}$.
    Then $\mathcal{G}_{\ess}$ is irreducible.
\end{proposition}
\begin{proof}
    It suffices to show that for any $v_1,v_2\in\mathcal{G}_{\ess}$, $1\leq i\leq d(v_1)$, $1\leq j\leq d(v_2)$, there exists a path $\eta$ from $v_1$ to $v_2$ such that $T(\eta)_{i,j}>0$.
    Let $\Delta\in\mathcal{F}$ be some net interval with $\vs(\Delta)=v_1$, let $\Delta_0\in\mathcal{F}_{t_0}$ be a net interval with neighbour set in the essential class such that $\rmax(\Delta_0)$ is maximal and $\#\vs(\Delta_0)$ maximal among such net intervals.
    Let $\phi$ be a path from $\vs(\Delta_0)$ to $v_2$.
    Let $\xi$ generate neighbour $i$ of $\Delta$ and let $\sigma\in\mathcal{I}^*$ have prefix $\xi$, $r_\sigma>0$ and $S_\sigma([0,1])\subseteq\Delta$; such a $\sigma$ must necessarily exist since $\Delta^\circ\cap S_\xi(K)\neq\emptyset$.
    Write $\vs(\Delta_0)=\{f_1,\ldots,f_k\}$ with $f_1<\cdots<f_k$ where each neighbour $f_j$ is generated by some word $\omega_j\in\Lambda_{t_0}$.
    We have $\Delta_0=[S_{\tau_1}(z_1),S_{\tau_2}(z_2)]$ where $\tau_1,\tau_2\in\Lambda_{t_0}$ and $z_1,z_2\in\{0,1\}$.

    We now show that $S_\sigma(\Delta_0)$ is indeed a net interval.
    Note that the words $\sigma\tau_j$ and $\sigma\omega_j$ are in $\Lambda_{r_\sigma t_0}$ by direct computation.
    Suppose for contradiction $S_\sigma(\Delta_0)$ is not a net interval.
    Without loss of generality, let $\omega_1$ have $|r_{\omega_1}|=R_{\max}(\Delta_0)\diam(\Delta_0)$.
    Since $S_{\omega_1}(K)\cap\Delta_0^\circ\neq\emptyset$, we have $S_{\sigma\omega_1}(K)\cap S_\sigma(\Delta_0)^\circ\neq\emptyset$ so there exists some net interval $\Delta'\subseteq S_\sigma(\Delta_0)$ where the inclusion is proper and $S_{\sigma\omega_1}(K)\cap(\Delta')^\circ\neq\emptyset$.
    But then $R_{\max}(\Delta')>R_{\max}(\Delta_0)$, contradicting the choice of $\Delta_0$.
    Thus $\Delta_1\coloneqq S_\sigma(\Delta_0)$ is indeed a net interval.

    Moreover, $\Delta_1$ has neighbours generated by the words $S_\sigma\circ S_{\omega_i}$, and since $r_\sigma>0$, $T_{\Delta_1}^{-1}=T_{\Delta_0}^{-1}\circ S_\sigma^{-1}$ and thus $\vs(S_\sigma(\Delta_0))\supseteq\vs(\Delta_0)$.
    Equality then follows by the maximality of $\#\vs(\Delta_0)$.

    As $\xi$ is a prefix of $\sigma$, write $\sigma=\xi\tau$ for some $\tau\in\mathcal{I}^*$.
    Let $\Delta$ have symbolic representation $\phi_0$.
    Since $\Delta_1\subseteq\Delta$, there exists $\eta_1$ such that $\Delta_1$ has symbolic representation $\phi_0\eta_1$.
    Since each neighbour of $\Delta_1$ is generated by a word $\sigma\omega_j=\xi\tau\omega_j$, by definition of the transition matrix and choice of $\xi$, row $i$ of the matrix $T(\eta_1)$ is strictly positive.
    But then $\eta\coloneqq\eta_1\phi$ is an admissible path from $v_1$ to $v_2$, and since every column of a transition matrix has a positive entry, row $i$ of the matrix $T(\eta)$ is strictly positive as well.
\end{proof}

\begin{remark}
    In fact, as we argued in the above proof, the essential class satisfies a somewhat stronger form of irreducibility: for any $v_1,v_2\in\mathcal{G}_{\ess}$ and $1\leq j\leq d(v_2)$, there exists a path $\eta$ from $v_1$ to $v_2$ such that row $j$ of $T(\eta)$ is strictly positive.
    This property is closely related to the key feature of the quasi-product structure under the weak separation condition demonstrated by Feng and Lau \cite{fl2009}.
    Moreover, under somewhat stronger hypotheses (satisfied, for example, when the attractor $K$ is an interval), the path $\eta$ can be chosen such that $T(\eta)$ is a positive matrix (see \cite{hhn2018} for a proof in the equicontractive case, but the general case follows similarly).
\end{remark}

\section{Sets of local dimensions of self-similar measures}\label{s:mf-properties}
We continue to use the notation of the previous section.
In particular, we assume that $\mathcal{G}$ is the matrix product system associated with an IFS that satisfies the finite neighbour condition.

\subsection{Basic results about local dimensions and periodic points}
The following notion is a well-studied way of quantifying the singularity of the measure $\mu$ with respect to Lebesgue measure at a point $x\in\supp\mu=K$.
\begin{definition}
    Let $x\in K$ be arbitrary.
    Then the \defn{lower local dimension of $\mu$ at $x$} is given by
    \begin{equation*}
        \underline{\dim}_{\loc}\mu(x)=\liminf_{t\to 0}\frac{\log \mu(B(x,t))}{\log t}
    \end{equation*}
    and the \defn{upper local dimension} is given similarly with the limit infimum replaced by the limit supremum.
    When the values of the upper and lower local dimension agree, we call the shared value the \defn{local dimension of $\mu$ at $x$}.
\end{definition}

Intuitively, the multifractal analysis of self-similar sets satisfying the finite neighbour condition is related to the multifractal analysis of the corresponding matrix product system.
However, the exact relationship is somewhat more complicated to establish: while the Lyapunov exponent of a path $\gamma$ depends only on the single sequence of edges determining $\gamma$, the local dimension of $\mus$ at a point $x\in K$ can also depend on net intervals which are adjacent to net intervals containing $x$.
This happens when $x$ is the shared boundary point of two distinct net intervals, but it can also happen when $x$ is approximated very well by boundary points (so that balls $B(x,r)$ overlap significantly with neighbouring net intervals, for many values of $r$).

A point $x\in K$ is said to be \defn{periodic} if it has a symbolic representation that is eventually a periodic path.
For such points, this issue with overlaps is easy to resolve.
A boundary point of a net interval is a periodic point and all elements of a simple loop class are periodic points.
Indeed, for each simple loop class there is a cycle $\theta$ such that all elements in the loop class have a symbolic representation of the form $\gamma_0\overline{\theta}$ where $\overline{\theta}$ is the infinite periodic path with cycle $\theta$.
If $x$ has two distinct symbolic representations, then $x$ is necessarily the endpoint of a net interval so the finite neighbour condition ensures that both symbolic representations are periodic points.
Note that a periodic point can be an interior point (in the sense of \cref{d:syr}), but every non-periodic point is interior.

By \cite[Proposition~3.15]{ruttoappear} (see also \cite[Proposition~2.7]{hhn2018}), we have the following simple formula for the local dimension of a periodic point:
\begin{proposition}[\cite{ruttoappear}, Proposition~3.15]\label{p:per-K}
    Suppose $x$ is an interior, periodic point with unique symbolic representation $\gamma$ which is eventually in the loop class $\mathcal{L}$.
    Let $\theta$ be any period of $\gamma$ and let $\overline{\theta}\in\Sigma(\mathcal{L})$ denote the path formed by repeating $\theta$ infinitely.
    Then the local dimension exists at $x$ and is given by
    \begin{equation*}
        \dim_{\loc}\mus(x)=\frac{\log\spr(T(\theta))}{\log W(\theta)}=\lambda(\mathcal{L},\overline{\theta}).
    \end{equation*}
    Otherwise, $x$ has two distinct symbolic representations with periods $\theta_1$ and $\theta_2$ and
    \begin{equation*}
        \dim_{\loc}\mus(x)=\min\left\{\frac{\log\spr(T(\theta_1))}{\log W(\theta_1)},\frac{\log\spr(T(\theta_2))}{\log W(\theta_2)}\right\}.
    \end{equation*}
\end{proposition}
\begin{corollary}
    If $x\in K_{\mathcal{L}}$ is a periodic point, then $\dim_{\loc}\mus(x)$ belongs to
    \begin{equation*}
        \bigcup_{i=1}^m\{\lambda(\mathcal{L}_i,\gamma):\gamma\in\Sigma(\mathcal{L}_i)\}
    \end{equation*}
    where $\mathcal{L}_1,\ldots,\mathcal{L}_m$ is a complete list of the maximal loop classes in $\mathcal{G}$.
\end{corollary}

More generally, when the Lyapunov exponent exists or the local dimension exists, we can relate the two notions.
\begin{proposition}\label{p:nper-dim}
    Suppose $x\in K$ is an interior point with exactly one symbolic representation $\gamma$ that is eventually in the maximal loop class $\mathcal{L}$.
       \begin{enumerate}[nl,r]
        \item Then
            \begin{equation*}
                \alpha_{\min}(\mathcal{L})\leq\underline{\lambda}(\mathcal{L},\gamma)\leq \overline{\dim}_{\loc}\mus(x)\leq\overline{\lambda}(\mathcal{L},\gamma)\leq\alpha_{\max}(\mathcal{L}).
            \end{equation*}
        \item If $\lambda(\mathcal{L},\gamma)$ exists, then $\overline{\dim}_{\loc}\mus(x)=\lambda(\gamma)$.
        \item If $\dim_{\loc}\mus(x)$ exists, then $\underline{\lambda}(\gamma)=\dim_{\loc}\mus(x)$.
    \end{enumerate}
\end{proposition}
\begin{proof}
   By assumption, $x$ belongs to $K_{\mathcal{L}}$ and $\mathcal{L}$ is unique with this property.
    \begin{enumerate}[nl,r]
        \item We have already seen that $\overline{\lambda}(\mathcal{L},\gamma)\leq\alpha_{\max}(\mathcal{L})$ and $\underline{\lambda}(\mathcal{L},\gamma)\geq\alpha_{\min}(\mathcal{L})$ in \cref{t:symb-lset}.

For any $t>0$, we have $B(x,t)\supseteq\Delta(\gamma|t)$ and therefore
            \begin{align*}
               \overline{\dim}_{\loc}\mus(x)= \limsup_{t\to 0}\frac{\log\mus(B(x,t))}{\log t}&\leq\limsup_{t\to 0}\frac{\log \norm{T(\gamma|t)}}{\log t}=\overline{\lambda}(\mathcal{L},\gamma).
            \end{align*}
            
            If $x$ is not a boundary point of some net interval, since there are only finitely many neighbour sets, there exists some $\rho>0$ and a monotonically increasing sequence $(n_k)_{k=1}^\infty$ with $n_1>N$ such that for each $k\in\N$ we have $ B(x,\rho W(\gamma|n_k))\subseteq\Delta(\gamma|n_k)$.
            Since $\rho W(\gamma|n_k)\asymp W(\gamma|n_k)$ we have
            \begin{align*}
                 \overline{\dim}_{\loc}\mus(x)=\limsup_{t\to 0}\frac{\log\mus(B(x,t))}{\log t} &\geq\limsup_{k\to\infty}\frac{\log\mus(B(x,\rho W(\gamma|n_k)))}{\log \rho W(\gamma|n_k)}\\
                                                                &\geq\limsup_{k\to\infty}\frac{\log \norm{T(\gamma|n_k)}}{\log W(\gamma|n_k)}\\
                                                                &\geq\underline{\lambda}(\mathcal{L},\gamma).                                                       
            \end{align*}
            as required.

            Otherwise, $x$ is a boundary point with a unique symbolic representation.
            In this case, for suitable $\rho>0$ and large $n$, $B(x,\rho W(\gamma|n))\cap K_{\mathcal{L}}\subseteq\Delta(\gamma|n)$, so we can argue similarly.

        \item This is immediate from (i) as $\underline{\lambda}(\mathcal{L},\gamma)=\overline{\lambda}(\mathcal{L},\gamma)$.
        \item The same argument as (i) shows that $\underline{\dim}_{\loc}\mus(x)\leq\underline{\lambda}(\mathcal{L},\alpha)$, from which the result follows.
    \end{enumerate}
\end{proof}

\subsection{Sets of local dimensions for simple and irreducible loop classes}
We begin by noting that periodic points are dense in the set of local dimensions.
\begin{proposition}
    Let $\mathcal{L}$ be an irreducible, maximal loop class that is not simple.
    Then the set of local dimensions at interior periodic points is dense in $[\alpha_{\min}(\mathcal{L}),\alpha_{\max}(\mathcal{L})]$.
\end{proposition}
\begin{proof}
    This follows by slightly modifying the proof of \cref{l:per-dens} by choosing the paths $\eta_k$ such that they are also interior paths.
    Thus the corresponding point in $K_{\mathcal{L}}$ is interior periodic and has local dimension equal to the symbolic local dimension by \cref{p:nper-dim}.
    Then the result follows from \cref{c:ly-set}, which states that $\{\underline{\lambda}(\mathcal{L}, \gamma):\gamma \text{ eventually in } \mathcal{L}\}=[\alpha_{\min}(\mathcal{L}), \alpha_{\max}(\mathcal{L})]$.
\end{proof}
Our next result establishes a converse to \cref{p:nper-dim}.
\begin{theorem}\label{t:gamma-proj}
    Let $\mathcal{L}$ be an irreducible, maximal loop class that is not simple.
    Then
    \begin{equation*}
        [\alpha_{\min}(\mathcal{L}),\alpha_{\max}(\mathcal{L})]\subseteq\{\dim_{\loc}\mus(x):x\in K_{\mathcal{L}},x\text{ interior}\}.
    \end{equation*}
    % Then $K(\mathcal{L},\alpha)\neq\emptyset$ if and only if $\Sigma(\mathcal{L},\alpha)\neq\emptyset$ if and only if $\alpha\in[\alpha_{\min}(\mathcal{L}),\alpha_{\max}(\mathcal{L})]$.
    % Moreover,
    % \begin{equation*}
        % \dim_H K(\mathcal{L},\alpha )\cap K_R\geq \dim_{H}\Sigma (\mathcal{L},\alpha )
    % \end{equation*}
    % for all $\alpha\in[\alpha_{\min}(\mathcal{L}),\alpha_{\max}(\mathcal{L})]$.
\end{theorem}
\begin{proof}
    Since $\mathcal{L}$ is not simple, there exists a path $\xi\in\Sigma^*(\mathcal{L})$ such that if $\xi$ is realized by $(\Delta_i)_{i=0}^m$, then $\Delta_m\subseteq\Delta_0^\circ$.
    Let $\alpha\in[\alpha_{\min}(\mathcal{L}),\alpha_{\max}(\mathcal{L})]$ be arbitrary and by \cref{t:symb-lset} get some $\gamma=(e_n)_{n=1}^\infty\in\Sigma(\mathcal{L})$ and a sequence $(n_j)_{j=1}^\infty$ with $\lim_j n_{j+1}/n_j=1$ such that $\lambda(\gamma)=\alpha$ and for each $j$, $\xi$ is a prefix of $(e_{n_j},e_{n_j+1},\ldots)$.
    Let $\zeta_0$ be such that $\zeta_0\gamma$ is a path in $\mathcal{G}$ beginning at the root vertex $\vs([0,1])$.
    By the choice of $\xi$, there exists a unique interior point $x$ with symbolic representation $\zeta_0\gamma$.
    We will show that $\dim_{\loc}\mus(x)=\alpha$.

    We first note that $B(x,t)\supseteq\Delta_t(x)$ where $\Delta_t(x)$ is the unique net interval in generation $t$ containing $x$.
    Thus if $\zeta_0e_1\ldots e_n$ is the symbolic representation of $\Delta_t(x)$, then
    \begin{equation*}
        \mus(B(x,t)))\geq \mus(\Delta_{t}(x))\asymp \norm{T(\zeta_0 e_1\ldots e_n)}\asymp \norm{T(e_1\ldots e_n)}
    \end{equation*}
    by \cref{l:left-prod}.
    Hence for some constant $C>0$ we have
    \begin{equation*}
        \frac{\log \mus(B(x,t))}{\log t}\leq\frac{\log C\norm{T(e_1\ldots e_n)}}{\log t}.
    \end{equation*}
    Since $\log t\asymp W(e_1\ldots e_n)$, it follows that
    \begin{equation*}
        \overline{\dim}_{\loc}\mus(x)\leq\alpha.
    \end{equation*}
    
    To obtain the other inequality, we use the special properties of the path $\gamma$.
    For each $k\in\N$, let $\Delta^{(k)}$ be the net interval with symbolic representation $\zeta_0 e_1\ldots e_{n_k-1}$.
    Let $\rho>0$ be such that for each $k\in\N$,
    \begin{equation*}
        B(x,\rho\diam(\Delta^{(k)}))\subseteq\Delta^{(k)}.
    \end{equation*}
    Given $t>0$ sufficiently small, let $k$ be such that $\rho\diam(\Delta^{(k+1)})\leq t<\rho\diam(\Delta^{(k)})$, so that
    \begin{equation*}
        B(x,t)\subseteq B(x,\rho\diam(\Delta^{(k)}))\subseteq\Delta^{(k)}.
    \end{equation*}%
    This ensures that $\mus(B(x,t))\leq\mus(\Delta^{(k)})$.
    We also have
    \begin{align*}
        t \geq \rho\diam(\Delta^{(k+1)}) &\asymp W(\zeta_0e_1\ldots e_{n_{k+1}-1})= W(\zeta_0 e_1\ldots e_{n_k-1})W(e_{n_k}\ldots e_{n_{k+1}-1})\\
                                         &\geq W_{\min}^{n_{k+1}-n_k} W(\zeta_0 e_1\ldots e_{n_k-1})\asymp W_{\min}^{n_{k+1}-n_k}\diam(\Delta^{(k)}).
    \end{align*}
    
    Combining these observations, we see that there exist positive constants $C_1, C_2$ such that
    \begin{align*}
        \frac{\log\mus(B(x,t))}{\log t} &\geq\frac{\log \mus(\Delta^{(k)})}{\log t}\geq\frac{\log C_1+\log\norm{T(e_1\ldots e_{n_k-1})}}{\log \rho+\log \diam(\Delta^{(k+1)})}\\
                                        &\geq\frac{\log C_1+\log\norm{T(e_1\ldots e_{n_k-1})}}{(n_{k+1}-n_k)\log C_2+\log\diam(\Delta^{(k)})}
    \end{align*}
    But $\log\diam(\Delta^{(k)})\asymp -n_k$ and $\lim_k(n_{k+1}-n_k)/n_k=0$ by choice of $\gamma$, so that
    \begin{align*}
        \underline{\dim}_{loc}\mus(x)&\geq\liminf_{k\to\infty}\frac{\log \norm{T(e_1\ldots e_{n_k-1})}}{\log \diam(\Delta^{(k)})}\\
                                     &\geq\liminf_{n\to\infty}\frac{\log\norm{T(e_1\ldots e_n)}}{\log W(e_1\ldots e_n)}=\alpha.
    \end{align*}
    We have thus shown that $\dim_{\loc}\mus(x)=\alpha$, as required.
\end{proof}
\begin{remark}
    If $\mathcal{L}$ is a simple loop class with interior points, it is clear that the conclusions of the theorem also hold and $\alpha_{\min}=\alpha_{\max}$.
\end{remark}
The preceding theorem gives us strong information about the set of attainable local dimensions:
\begin{corollary}\label{c:loop-set}
    Let $\mathcal{L}$ be an irreducible, maximal loop class that is not simple.
    Then
    \begin{align*}
        [\alpha_{\min}(\mathcal{L}),\alpha_{\max}(\mathcal{L})] &= \{\dim_{\loc}\mus(x):x\in K_{\mathcal{L}},x\text{ interior}\}\\
                                                                &= \{\overline{\dim}_{\loc}\mus(x):x\in K_{\mathcal{L}},x\text{ interior}\}.
    \end{align*}
\end{corollary}
\begin{proof}
    This follows by combining \cref{t:gamma-proj} and \cref{p:nper-dim}, noting that the set of local dimensions is contained in the set of upper local dimensions.
\end{proof}
\begin{remark}\label{r:ldim-exception}
    When $\mathcal{L}$ is the essential class, this result was shown in \cite{ruttoappear}.
    Moreover, in that case, $[\alpha_{\min},\alpha_{\max}]$ is also the set of lower local dimensions.

    However, outside the essential class, the same statement need not hold for the lower local dimension in place of the upper local dimension; the set of lower local dimensions can be strictly larger.
    Consider the example from \cref{ex:other}; this example and the result here is treated in \cite[Section~6]{hhn2018}.
    In that example, with our notation, if $\mathcal{L}$ is the irreducible maximal loop class not equal to the essential class, then
    \begin{equation*}
        [\alpha_{\min}(\mathcal{L}),\alpha_{\max}(\mathcal{L})]=\left[\frac{\log 7}{\log 4},\frac{\log 14}{\log 4}\right],
    \end{equation*}
    while 
    \begin{equation*}
        \{\underline{\dim }_{loc}\mus(x):x\in K_{\mathcal{L}}\text{ interior}\}=\left[\frac{1}{2},\frac{\log 14}{\log 4}\right].
    \end{equation*}
    Here $1/2$ is also the local dimension of a boundary point (that is not interior).

    It would be interesting to know if the set of lower local dimensions at interior points is always an interval.
\end{remark}

If $\mathcal{L}$ is a loop class that contains interior points, then the set of local dimensions at these interior points is given by the interval $[\alpha_{\min}(\mathcal{L}),\alpha_{\max}(\mathcal{L})]$.
If $\mathcal{L}$ does not contain any interior points, then $\mathcal{L}$ is a simple loop class.
In this situation, it may hold that every $x\in K_{\mathcal{L}}$ has two symbolic representations, and the local dimension is always given by the symbolic representation of the adjacent path which is not eventually in $\mathcal{L}$.
This motivates the following definition:
\begin{definition}
    We say that a loop class $\mathcal{L}$ is \defn{non-degenerate} if $\mathcal{L}$ is not simple, or if $\mathcal{L}$ is simple with period $\theta$ and there exists some $x\in K_{\mathcal{L}}$ such that
    \begin{equation*}
        \dim_{\loc}\mus(x)=\lambda(\overline{\theta}).
    \end{equation*}
    We say that $\mathcal{L}$ is \defn{degenerate} otherwise.
\end{definition}
We emphasize that, unlike simplicity or irreducibility, degeneracy depends on the choice of probabilities.
For an example of this phenomenon, see \cref{ex:ftype}.

The point is that if $\mathcal{L}$ is a degenerate loop class, then the local dimension at any point $x\in K_{\mathcal{L}}$ is given by the Lyapunov exponent of a path not in $\mathcal{L}$.
We now have the following corollary, which holds under the assumptions that all maximal loop classes are either irreducible or simple.
\begin{corollary}\label{c:attainable}
    Let $\{S_i\}_{i\in\mathcal{I}}$ be an IFS satisfying the finite neighbour condition with maximal loop classes $\{\mathcal{L}_i\}_{i=1}^\ell$.
    Suppose each $\mathcal{L}_i$ is either irreducible or simple.
    Let $\{\mathcal{L}_i\}_{i=1}^{\ell'}$ denote the non-degenerate maximal loop classes.
    Then
    \begin{equation*}
        \{\overline{\dim}_{\loc}\mus(x):x\in K\}=\{\dim_{loc}\mus(x):x\in K\}=\bigcup_{i=1}^{\ell'}[\alpha_{\min}(\mathcal{L}_i),\alpha_{\max}(\mathcal{L}_i)].
    \end{equation*}
\end{corollary}
\begin{proof}
    If $x$ is a periodic point, then $\dim_{\loc}\mus(x)=\lambda(\mathcal{L},\gamma)$ for some path $\gamma$ in some non-degenerate loop class $\mathcal{L}$ according to \cref{p:per-K}.
    Otherwise, $x$ must be an interior point of some $K_{\mathcal{L}}$ where $\mathcal{L}$ is not simple.
    By \cref{c:loop-set}, $\overline{\dim}_{\loc}\mus(x)=\alpha$ for some $\alpha\in[\alpha_{\min}(\mathcal{L}),\alpha_{\max}(\mathcal{L})]$.

    On the other hand, if $\mathcal{L}$ is a non-simple loop class, then \cref{t:gamma-proj} shows that each $\alpha \in [\alpha_{\min}(\mathcal{L}),\alpha_{\max}(\mathcal{L})]$ is attained as a local dimension.
    If $\mathcal{L}$ is a simple loop class, then $\alpha_{\min}(\mathcal{L})=\alpha_{\max}(\mathcal{L})$ and this value is attained as a local dimension precisely when $\mathcal{L}$ is non-degenerate.
\end{proof}
\begin{remark}
    The authors do not know if this result continues to hold without the irreducibility assumption.
    However, we are not aware of any examples in $\R$ satisfying the weak separation condition which do not satisfy the hypotheses for \cref{c:attainable}.
\end{remark}

\section{Examples of IFS satisfying the finite neighbour condition}\label{s:exifs}
Throughout this section, for $\mathcal{L}$ a maximal loop class and $\mus$ a self similar measure, we will write
\begin{align*}
    \mathcal{D}(\mathcal{L})&=\{\lambda(x):x\in\Sigma(\mathcal{L})\}\\
    \mathcal{D}(\mus) &= \{\dim_{\loc}\mus(x):x\in\supp\mus\}.
\end{align*}

\subsection{Bernoulli convolutions}\label{ex:bconv}
One much studied example of an equicontractive IFS of finite type is the IFS with two contractions, 
\begin{equation}
    \{\rho x,\rho x+1-\rho \},  \label{BCIFS}
\end{equation}
with $\rho$ the reciprocal of the Golden mean.
Feng \cite{fen2005} (see also \cite{hhm2016,hhn2018}) computed the neighbour sets (or characteristic vectors in his terminology) with respect to the original net interval construction.

In our slightly modified setting, there are six neighbour sets.
These are:
\begin{align*}
    v_1 &=\{x\} & v_2 &= \{x\cdot(1+\rho)\}\\
    v_3 &=\{x\cdot(1+\rho)-\rho\} & v_{4} &=\{x\cdot(2+\rho),x(2+\rho)-(1+\rho)\}\\
    v_{5}&=\{x\cdot(3+2\rho)-(1+\rho)\} & v_6&=\{x\cdot(1+\rho), x\cdot(1+\rho)-\rho\}
\end{align*}
The weight function is given by $W(e)=\rho $ for all edges $e$.
The essential class has $V(\mathcal{G}_{\ess})=\{v_4,v_5,v_6\}$ and there are two other maximal loop classes, $\mathcal{L}_1$ and $\mathcal{L}_2$, which are the simple loops with vertex sets $V(\mathcal{L}_1)=\{v_{2}\}$ and $V(\mathcal{L}_2)=\{v_{3}\}$.
Both simple loop classes are non-degenerate since 0 and 1 are interior points.
We have $K_{\ess}=(0,1)$, $K_{\mathcal{L}_1}=\{0\}$, and $K_{\mathcal{L}_2}=\{1\}$.
See \cref{f:gm-graph} for the transition graph as well as the associated transition matrices.
\begin{figure}[ht]
    \begin{tikzpicture}[
        baseline=(current bounding box.center),
        vtx/.style={circle,inner sep=1pt,draw=black},
        elbl/.style={fill=white,circle,inner sep=1pt},
        edge/.style={thick,->,>=stealth},
        l1v/.style={fill=green!30},
        l1e/.style={draw=green!50!black},
        l2v/.style={fill=blue!30},
        l2e/.style={draw=blue!50!black},
        ev/.style={fill=red!30},
        ee/.style={draw=red!50!black},
        ]
        \node[vtx] (v1) at (0,6) {$v_1$};
        \node[vtx, l1v] (v2) at (-2,4) {$v_2$};
        \node[vtx, l2v] (v3) at (2,4) {$v_3$};
        \node[vtx, ev] (v4) at (0,2) {$v_4$};
        \node[vtx, ev] (v5) at (2.5,-0.5) {$v_5$};
        \node[vtx, ev] (v6) at (-2.5,-0.5) {$v_6$};

        \draw[edge] (v1) -- node[elbl]{$e_1$}(v2);
        \draw[edge] (v1) -- node[elbl]{$e_2$}(v3);
        \draw[edge] (v1) -- node[elbl]{$e_3$}(v4);

        \draw[edge,l1e] (v2) .. controls +(135:2) and +(225:2) .. node[elbl]{$e_4$}(v2);
        \draw[edge] (v2) -- node[elbl]{$e_5$}(v4);

        \draw[edge,l2e] (v3) .. controls +(45:2) and +(-45:2) .. node[elbl]{$e_6$}(v3);
        \draw[edge] (v3) -- node[elbl]{$e_7$}(v4);

        \draw[edge,ee] (v5) -- node[elbl]{$e_{8}$}(v4);
        \draw[edge,ee] (v6) -- node[elbl]{$e_{9}$}(v5);

        \draw[edge,ee] (v6) -- node[elbl]{$e_{10}$} (v4);
        \draw[edge,ee] (v6) to[out=45-30,in=225+30] node[elbl]{$e_{11}$} (v4);
        \draw[edge,ee] (v4) to[out=225-30,in=45+30] node[elbl]{$e_{12}$} (v6);

    \end{tikzpicture}
    \begin{tabular}{ccc}
        Edge & Weight & Transition Matrix\\
        $e_1$ & $\rho$ & $\begin{pmatrix}p\end{pmatrix}$\\
        $e_2$ & $\rho$ & $\begin{pmatrix}1-p\end{pmatrix}$\\
        $e_3$ & $\rho$ & $\begin{pmatrix}1-p&p\end{pmatrix}$\\
        $e_4$ & $\rho$ & $\begin{pmatrix}p\end{pmatrix}$\\
        $e_5$ & $\rho$ & $\begin{pmatrix} 1-p&p\end{pmatrix}$\\
        $e_6$ & $\rho$ & $\begin{pmatrix} 1-p\end{pmatrix}$\\
        $e_7$ & $\rho$ & $\begin{pmatrix} 1-p&p\end{pmatrix}$\\
        $e_8$ & $\rho$ & $\begin{pmatrix}1-p&p\end{pmatrix}$\\
        $e_9$ & $\rho$ & $\begin{pmatrix}p\\1-p\end{pmatrix}$\\
        $e_{10}$ & $\rho$ & $\begin{pmatrix}p & 0\\1-p&p\end{pmatrix}$\\
        $e_{11}$ & $\rho$ & $\begin{pmatrix}1-p&p\\0&1-p\end{pmatrix}$\\
        $e_{12}$ & $\rho$ & $\begin{pmatrix}p&0\\0&1-p\end{pmatrix}$
    \end{tabular}
    \caption{Transition graph for the Golden mean Bernoulli convolution}
    \label{f:gm-graph}
\end{figure}
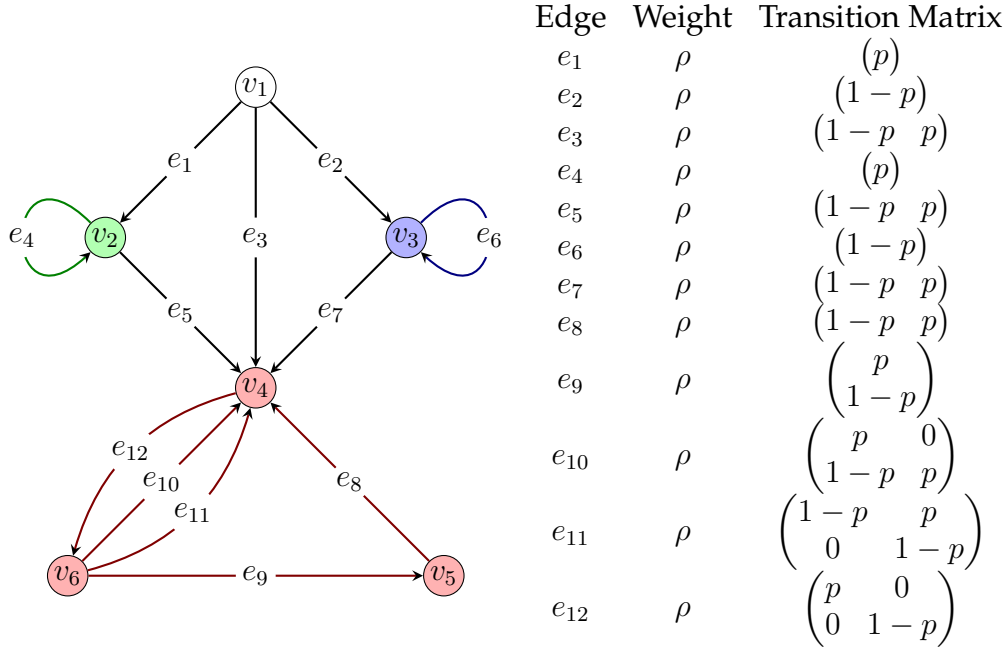

Since the essential class is always irreducible and the non-essential maximal loop classes are simple, the set of local dimensions is a union of a possibly non-singleton interval along with at most two isolated points.
The corresponding sets of Lyapunov exponents are
\begin{equation*}
    \mathcal{D}(\mathcal{L}_1)=\left\{\frac{\log p}{\log\rho}\right\}\text{ and }\mathcal{D}(\mathcal{L}_2)=\left\{\frac{\log(1-p)}{\log\rho}\right\}.
\end{equation*}
These are also the local dimensions at $0$ and $1$ respectively since $0$ and $1$ are interior points, so we have $\mathcal{D}(\mathcal{L}_1),\mathcal{D}(\mathcal{L}_2)\subseteq\mathcal{D}(\mus)$.

Now, for $p\leq 1/2$, note that the transition matrix of the cycle $(e_{11},e_{12})$ has spectral radius $(1-p)^2$ and weight $\rho^2$.
The Lyapunov exponent corresponding to this path is $\frac{\log(1-p)}{\log\rho}$ so that $\mathcal{D}(\mathcal{L}_2)\subseteq \mathcal{D}(\mathcal{G}_{\ess})$.
Similarly if $p\geq 1/2$, then the cycle $(e_{10},e_{12})$ has corresponding Lyapunov exponent $\frac{\log p}{\log\rho}$ and $\mathcal{D}(\mathcal{L}_1)\subseteq \mathcal{D}(\mathcal{G}_{\ess})$.
In particular, when $p=1/2$, then $\mathcal{D}(\mus)$ is a closed interval, and when $p\neq 1/2$, $\mathcal{D}(\mus)$ is a closed interval along with at most a singleton point.

When $p\neq 1/2$, we know in general, by a short argument in \cite{hh2019}, that $\mathcal{D}(\mus)$ must contain an isolated point corresponding to either $x=0$ or $x=1$, so $\mathcal{D}(\mus)$ is precisely a closed interval along with an isolated point.

\subsection{Testud measures}\label{ex:testud}
Consider the IFS given by the maps
\begin{align*}
    S_1(x)&=\frac{x}{4}& S_2(x)&=\frac{x}{4}+\frac{1}{4}&S_3(x)&=\frac{x}{4}+\frac{1}{2}& S_4(x)&=\frac{x}{4}+\frac{3}{4}\\
    S_5(x)&=-\frac{x}{4}+\frac{1}{4}& S_6(x)&=-\frac{x}{4}+\frac{1}{2}
\end{align*}
This example is treated in \cite[Section~6.2]{tes2006a}.
For each $i$, we have $S_i([0,1])=[j/4,(j+1)/4]$ for some $j\in\{0,1,2,3\}$.
There are two neighbour sets,
\begin{align*}
    v_1 &= \{x\} & v_2&=\{-x+1,x\}.
\end{align*}
The transition graph is given in \cref{f:testud}, and there is the essential class $\mathcal{G}_{\ess}$ with vertex set $\{v_2\}$ and a non-simple irreducible maximal loop class $\mathcal{L}$ with vertex set $\{v_1\}$.

Every cycle in $\mathcal{L}$ is a concatenation of the edges $e_1$ and $e_2$; since the corresponding transition matrices are singletons, we have
\begin{align*}
    \mathcal{D}(\mathcal{L})&=\Bigl\{\frac{\log(p_3^np_4^m)}{-\log 4^{n+m}}:n\geq 0,m\geq 0,n+m\geq 1\Bigr\}\\
                            &=\Bigl[\frac{\log \max\{p_3,p_4\}}{-\log 4},\frac{\log \min\{p_3,p_4\}}{-\log 4}\Bigr].
\end{align*}
Similarly, the cycles in $\mathcal{G}_{\ess}$ are arbitrary concatenations of edges in $\{e_5,e_6,e_7,e_8\}$.
Now, under the assumption that $p_1=p_4+p_5$ and $p_2=p_3+p_6$, if $\eta=(e_{i_1},e_{i_2},\ldots,e_{i_k})$ is any path in the essential class with $n$ edges in $\{e_5,e_8\}$ and $m$ edges in $\{e_6,e_7\}$, one may show that $\spr T(\eta)=p_1^np_2^m$.
Thus
\begin{align*}
    \mathcal{D}(\mathcal{G}_{\ess})&=\Bigl\{\frac{\log(p_1^np_2^m)}{-\log 4^{n+m}}:n\geq 0,m\geq 0,n+m\geq 1\Bigr\}\\
                                   &=\Bigl[\frac{\log \max\{p_1,p_2\}}{-\log 4},\frac{\log \min\{p_1,p_2\}}{-\log 4}\Bigr].
\end{align*}
We therefore have
\begin{align*}
    \mathcal{D}(\mus)=\Bigl[\frac{\log \max\{p_1,p_2\}}{-\log 4},\frac{\log \min\{p_1,p_2\}}{-\log 4}\Bigr]\cup\Bigl[\frac{\log \max\{p_3,p_4\}}{-\log 4},\frac{\log \min\{p_3,p_4\}}{-\log 4}\Bigr].
\end{align*}
In particular, if $p_1<p_2<p_3<p_4$, then $\mathcal{D}(\mus)$ is a disjoint union of two non-trivial closed intervals.

Note that the other examples treated in \cite{tes2006a} can be analyzed similarly.
\begin{figure}[ht]
    \begin{tikzpicture}[
        baseline=(current bounding box.center),
        vtx/.style={circle,inner sep=1pt,draw=black},
        elbl/.style={fill=white,circle,inner sep=1pt},
        edge/.style={thick,->,>=stealth},
        l1v/.style={fill=green!30},
        l1e/.style={draw=green!50!black},
        l2v/.style={fill=blue!30},
        l2e/.style={draw=blue!50!black},
        ev/.style={fill=red!30},
        ee/.style={draw=red!50!black},
        ]
        \node[vtx,l1v] (v1) at (4,0) {$v_1$};
        \node[vtx,ev] (v2) at (0,0) {$v_2$};

        % self loops at v1
        \draw[edge,l1e] (v1) .. controls +({60+45}:2) and +({60-45}:2) .. node[elbl]{$e_1$} (v1);
        \draw[edge,l1e] (v1) .. controls +({-60+45}:2) and +({-60-45}:2) .. node[elbl]{$e_2$} (v1);

        % double edges joining
        \draw[edge] (v1) to[out={180+12},in={-12}] node[elbl]{$e_3$} (v2);
        \draw[edge] (v1) to[out={180-12},in={12}] node[elbl]{$e_4$} (v2);

        % self loops at v2
        \draw[edge,ee] (v2) .. controls +({72+18}:3) and +({72-18}:3) .. node[elbl]{$e_8$} (v2);
        \draw[edge,ee] (v2) .. controls +({2*72+18}:3) and +({2*72-18}:3) .. node[elbl]{$e_7$} (v2);
        \draw[edge,ee] (v2) .. controls +({3*72+18}:3) and +({3*72-18}:3) .. node[elbl]{$e_6$} (v2);
        \draw[edge,ee] (v2) .. controls +({4*72+18}:3) and +({4*72-18}:3) .. node[elbl]{$e_5$} (v2);
        % \draw[edge] (v1) -- node[elbl]{$e_2$}(v2);
        % \draw[edge] (v1) -- node[elbl]{$e_4$}(v4);
        % \draw[edge] (v2) -- node[elbl]{$e_6$}(v3);
        % \draw[edge] (v2) -- node[elbl]{$e_7$}(v4);
        % \draw[edge] (v5) -- node[elbl]{$e_{10}$}(v2);
    \end{tikzpicture}
    \begin{tabular}{ccc}
        Edge & Weight & Transition Matrix\\
        $e_1$ & $1/4$ & $\begin{pmatrix}p_3\end{pmatrix}$\\
        $e_2$ & $1/4$ & $\begin{pmatrix}p_4\end{pmatrix}$\\
        $e_3$ & $1/4$ & $\begin{pmatrix}p_5&p_1\end{pmatrix}$\\
        $e_4$ & $1/4$ & $\begin{pmatrix}p_6&p_2\end{pmatrix}$\\
        $e_5$ & $1/4$ & $\begin{pmatrix}p_4&0\\p_5&p_1\end{pmatrix}$\\
        $e_6$ & $1/4$ & $\begin{pmatrix}p_3&0\\p_6&p_2\end{pmatrix}$\\
        $e_7$ & $1/4$ & $\begin{pmatrix}p_2&p_6\\0&p_3\end{pmatrix}$\\
        $e_8$ & $1/4$ & $\begin{pmatrix}p_1&p_5\\0&p_4\end{pmatrix}$
    \end{tabular}
    \caption{Transition graph for the Testud IFS}
    \label{f:testud}
\end{figure}
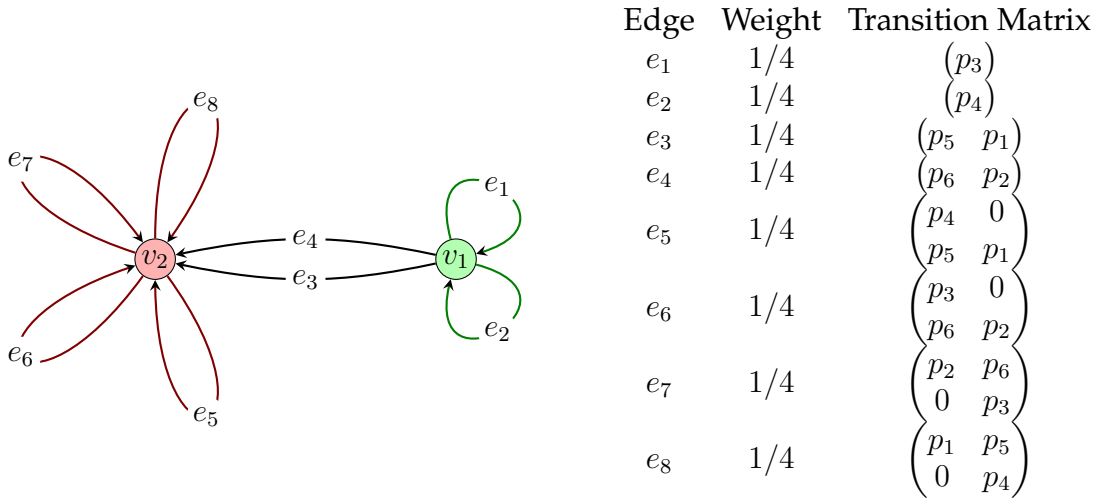
\subsection{Other examples}
\subsubsection{Cantor-like measures}
Consider the family of IFS given by
\begin{equation*}
    \bigl\{S_j(x)=\frac{x}{r}+\frac{j}{mr}(r-1):0\leq j\leq m\bigr\}
\end{equation*}
for integers $m,r$ satisfying $2\leq r\leq m$.
This family includes a rescaled version of the three-fold convolution of the middle-third Cantor measure, which was the earliest example of a self-similar measure known to exhibit isolated points in the set of local dimensions \cite{hl2001}.
The transition graph consists of an essential class along with two simple maximal loop classes $\mathcal{L}_1$ and $\mathcal{L}_2$, where $K_{\mathcal{L}_1}=\{0\}$, $K_{\mathcal{L}_2}=\{1\}$, and $K_{\mathcal{G}_{\ess}}=(0,1)$.
The loops $\mathcal{L}_1$ and $\mathcal{L}_2$ consist of single edges with $1\times 1$ transition matrices, and
\begin{align*}
    \mathcal{D}(\mathcal{L}_1)&=\frac{\log p_0}{-\log r} & \mathcal{D}(\mathcal{L}_2) &= \frac{\log p_m}{-\log r}.
\end{align*}
For appropriately chosen probabilities, these singletons contribute the isolated points in the set of local dimensions for the self-similar measure $\mus$ and the essential class contributes a closed interval of dimensions.
See \cite{hhn2018} for more details.

\subsubsection{An example of Lau and Wang}\label{ex:lw}
By nature of the definition, an IFS of finite type must have logarithmically commensurable contraction factors.
Here is an example of an IFS satisfying the finite neighbour condition which does not have commensurable contraction factors.

The IFS $\{\rho x,rx+\rho (1-r),rx+1-r\}$ where $\rho +2r-\rho r\leq 1$ was seen to satisfy the WSC in \cite{lw2004}, but it is not of finite type when $\rho$ and $r$ are non-commensurable.
For simplicity, we consider the case $\rho=1/3$ and $r=1/4$; for a more general treatment, this family was studied in \cite[Section~5.2]{ruttoappear}.

There are 5 neighbour sets given by
\begin{align*}
    v_1 &= \{x\} & v_2 &= \{4x/3\}\\
    v_3 &= \{3x/2-1/2\} & v_4 &= \{3x,4x-3\}\\
    v_5 &= \{x,3x\}.
\end{align*}
The transition graph and transition matrices are given in \cref{f:tgraph}.
One can see that there is only one maximal loop class, which is the essential class; thus, the set of local dimensions is a closed interval.
For more details on the computations of the set of attainable local dimensions, we refer the reader to \cite{ruttoappear}.
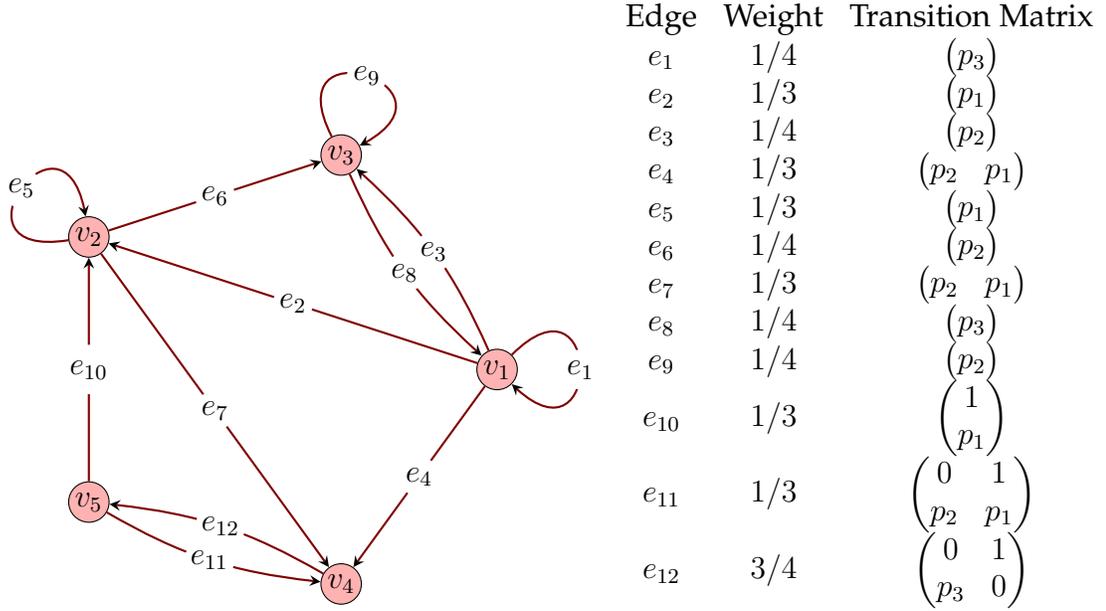
\begin{figure}[ht]
    \begin{tikzpicture}[
        baseline=(current bounding box.center),
        vtx/.style={circle,inner sep=1pt,draw=black,fill=red!30},
        elbl/.style={fill=white,circle,inner sep=1pt},
        edge/.style={thick,->,>=stealth},
        ev/.style={fill=red!30},
        ee/.style={draw=red!50!black},
        ]
        \node[vtx,ev] (v1) at (0:3) {$v_1$};
        \node[vtx,ev] (v3) at (72:3) {$v_3$};
        \node[vtx,ev] (v2) at (144:3) {$v_2$};
        \node[vtx,ev] (v5) at (216:3) {$v_5$};
        \node[vtx,ev] (v4) at (288:3) {$v_4$};

        % self loops
        \draw[edge,ee] (v1) .. controls +({0+45}:2) and +({0-45}:2) .. node[elbl]{$e_1$} (v1);
        \draw[edge,ee] (v3) .. controls +({72+45}:2) and +({72-45}:2) .. node[elbl]{$e_9$} (v3);
        \draw[edge,ee] (v2) .. controls +({144+45}:2) and +({144-45}:2) .. node[elbl]{$e_5$} (v2);

        % double edges
        \draw[edge,ee] (v1) to[out={126+0*72-12},in={306+0*72+12}] node[elbl]{$e_3$} (v3);
        \draw[edge,ee] (v3) to[in={126+0*72+12},out={306+0*72-12}] node[elbl]{$e_8$} (v1);
        \draw[edge,ee] (v5) to[out={126+3*72-12},in={306+3*72+12}] node[elbl]{$e_{11}$} (v4);
        \draw[edge,ee] (v4) to[in={126+3*72+12},out={306+3*72-12}] node[elbl]{$e_{12}$} (v5);

        % single edges
        \draw[edge,ee] (v1) -- node[elbl]{$e_2$}(v2);
        \draw[edge,ee] (v1) -- node[elbl]{$e_4$}(v4);
        \draw[edge,ee] (v2) -- node[elbl]{$e_6$}(v3);
        \draw[edge,ee] (v2) -- node[elbl]{$e_7$}(v4);
        \draw[edge,ee] (v5) -- node[elbl]{$e_{10}$}(v2);
    \end{tikzpicture}
    \begin{tabular}{ccc}
        Edge & Weight & Transition Matrix\\
        $e_1$ & $1/4$ & $\begin{pmatrix}p_3\end{pmatrix}$\\
        $e_2$ & $1/3$ & $\begin{pmatrix}p_1\end{pmatrix}$\\
        $e_3$ & $1/4$ & $\begin{pmatrix}p_2\end{pmatrix}$\\
        $e_4$ & $1/3$ & $\begin{pmatrix}p_2&p_1\end{pmatrix}$\\
        $e_5$ & $1/3$ & $\begin{pmatrix}p_1\end{pmatrix}$\\
        $e_6$ & $1/4$ & $\begin{pmatrix}p_2\end{pmatrix}$\\
        $e_7$ & $1/3$ & $\begin{pmatrix}p_2&p_1\end{pmatrix}$\\
        $e_8$ & $1/4$ & $\begin{pmatrix}p_3\end{pmatrix}$\\
        $e_9$ & $1/4$ & $\begin{pmatrix}p_2\end{pmatrix}$\\
        $e_{10}$ & $1/3$ &$\begin{pmatrix}1\\p_1\end{pmatrix}$\\
        $e_{11}$ & $1/3$ & $\begin{pmatrix}0&1\\p_2&p_1\end{pmatrix}$\\
        $e_{12}$ & $3/4$ & $\begin{pmatrix}0&1\\p_3&0\end{pmatrix}$
    \end{tabular}
    \caption{Transition graph for the example of Lau and Wang}
    \label{f:tgraph}
\end{figure}

\subsubsection{A non-equicontractive finite type example}\label{ex:ftype}
Here is an example which satisfies the finite type condition without equal contraction ratios.

Take $\rho=(\sqrt{5}-1)/2$, the reciprocal of the Golden mean.
Consider the IFS given by the maps
\begin{align*}
    S_1(x) &= \rho x & S_2(x)&=\rho^2 x+\rho-\rho^2 & S_3(x) &= \rho^2 x+(1-\rho^2)
\end{align*}
with probabilities $(p_i)_{i=1}^3$.
This IFS has 7 neighbour sets given by
\begin{align*}
    v_1 &= \{x\}\\
    v_2 &= \{(2+\rho)x\}\\
    v_3 &= \{x,(1+\rho)x-(1+\rho)\}\\
    v_4 &=\{(2+\rho)x,(3+2\rho)x-(1+\rho)\}\\
    v_5 &= \{(1+\rho)x-\rho,(2+\rho)x,(2+\rho)x-(1+\rho)\} &&\\
    v_6 &= \{(1+\rho)x,(2+\rho)x,(2+\rho)x-1\} &&\\
    v_7 &= \{(2+\rho)x,(2+\rho)x-(1+\rho),(3+2\rho)x-2(1+\rho),(3+2\rho)x-(1+\rho)\}.
\end{align*}
There are three simple non-essential maximal loop classes, with vertex sets $V(\mathcal{L}_1)=\{v_1\}$, $V(\mathcal{L}_2)=\{v_2\}$, and $V(\mathcal{L}_3)=\{v_3\}$.
The essential class has vertex set $V(\mathcal{G}_{\ess})=\{v_4,v_5,v_6,v_7\}$.
The transition graph and transition matrices are given in \cref{f:ftype}.

A direct computation shows that $\mathcal{D}(\mathcal{L}_1)=\mathcal{D}(\mathcal{L}_3)=\frac{\log p_3}{2\log\rho}$ and $\mathcal{D}(\mathcal{L}_2)=\frac{\log p_1}{\log\rho}$.
Thus $\mathcal{D}(\mus)$ consists of a possibly non-singleton interval along with at most two isolated points.
Both $K_{\mathcal{L}_1}$ and $K_{\mathcal{L}_2}$ contain interior points, so they are non-degenerate.
However, every point in $x\in K_{\mathcal{L}_3}$ has two symbolic representations of the form
\begin{equation*}
    (\underbrace{e_1,\ldots,e_1}_n,e_3,e_6,e_6,\ldots)\text{ and }(\underbrace{e_1,\ldots,e_1}_n,e_1,e_2,e_4,e_4,\ldots)
\end{equation*}
for some $n\geq 0$.
Thus for any $x\in K_{\mathcal{L}_3}$, we have
\begin{equation*}
    \dim_{\loc}\mus(x) = \min\Bigl\{\frac{\log p_1}{\log\rho},\frac{\log p_3}{2\log\rho}\Bigr\}
\end{equation*}
and when the minimum is not attained at $\frac{\log p_3}{2\log\rho}$, $\mathcal{L}_3$ is a degenerate loop class.
However, this does not impact the set of possible local dimensions.

Suppose in particular that the probabilities satisfy $p_1^2>p_2$ and $p_3>p_2$.
Then the cycle $(e_{10},e_{11})$ in the essential class has $\spr T(e_{10},e_{11})=p_1^2$ and $W(e_{10},e_{11})=\rho^2$, so
\begin{equation*}
    \mathcal{D}(\mathcal{L}_2)\subseteq\mathcal{D}(\mathcal{G}_{\ess}).
\end{equation*}
Similarly, the cycle $(e_{12},e_{13})$ has $\spr T(e_{12},e_{13})=p_3$ and $W(e_{12},e_{13})=\rho^2$ so
\begin{equation*}
    \mathcal{D}(\mathcal{L}_1)=\mathcal{D}(\mathcal{L}_3)\subseteq\mathcal{D}(\mathcal{G}_{\ess})
\end{equation*}
Therefore $\mathcal{D}(\mus)=\mathcal{D}(\mathcal{G}_{\ess})$ is a closed interval for such probabilities.
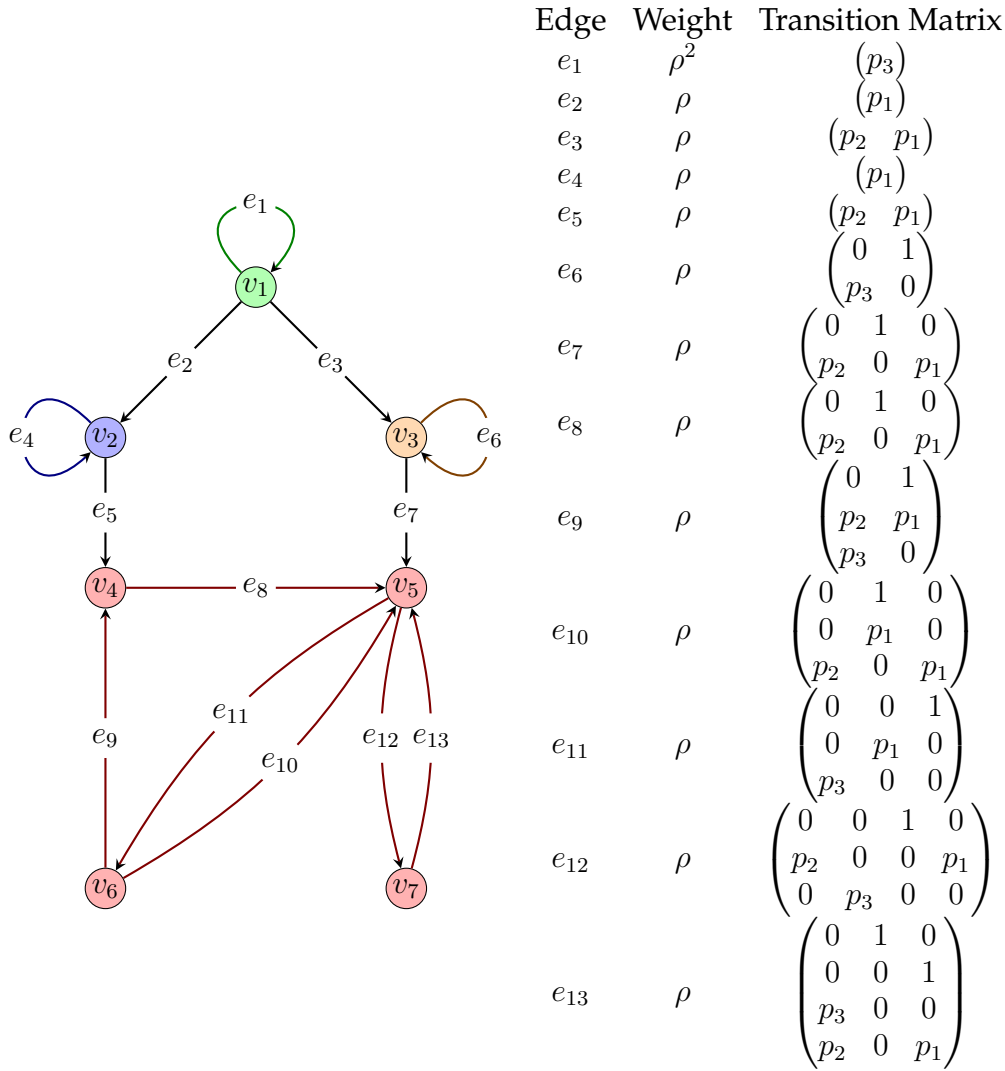
\begin{figure}[ht]
    \begin{tikzpicture}[
        baseline=(current bounding box.center),
        vtx/.style={circle,inner sep=1pt,draw=black},
        elbl/.style={fill=white,circle,inner sep=1pt},
        edge/.style={thick,->,>=stealth},
        l1v/.style={fill=green!30},
        l1e/.style={draw=green!50!black},
        l2v/.style={fill=blue!30},
        l2e/.style={draw=blue!50!black},
        l3v/.style={fill=orange!30},
        l3e/.style={draw=orange!50!black},
        ev/.style={fill=red!30},
        ee/.style={draw=red!50!black},
        ]
        \node[vtx, l1v] (v1) at (2,2) {$v_1$};
        \node[vtx, l2v] (v2) at (0,0) {$v_2$};
        \node[vtx, l3v] (v3) at (4,0) {$v_3$};
        \node[vtx, ev] (v4) at (0,-2) {$v_4$};
        \node[vtx, ev] (v5) at (4,-2) {$v_5$};
        \node[vtx, ev] (v6) at (0,-6) {$v_6$};
        \node[vtx, ev] (v7) at (4,-6) {$v_7$};

        \draw[edge,l1e] (v1) .. controls +(90+45:2) and +(90-45:2) .. node[elbl]{$e_1$}(v1);
        \draw[edge] (v1) -- node[elbl]{$e_2$}(v2);
        \draw[edge] (v1) -- node[elbl]{$e_3$}(v3);

        \draw[edge,l2e] (v2) .. controls +(135:2) and +(225:2) .. node[elbl]{$e_4$}(v2);
        \draw[edge] (v2) -- node[elbl]{$e_5$}(v4);

        \draw[edge,l3e] (v3) .. controls +(45:2) and +(-45:2) .. node[elbl]{$e_6$}(v3);
        \draw[edge] (v3) -- node[elbl]{$e_7$}(v5);

        \draw[edge,ee] (v4) -- node[elbl]{$e_8$}(v5);
        \draw[edge,ee] (v6) -- node[elbl]{$e_9$}(v4);

        \draw[edge,ee] (v6) to[out=45-15,in=225+15] node[elbl]{$e_{10}$} (v5);
        \draw[edge,ee] (v5) to[out=225-15,in=45+15] node[elbl]{$e_{11}$} (v6);

        \draw[edge,ee] (v5) to[out=270-15,in=90+15] node[elbl]{$e_{12}$} (v7);
        \draw[edge,ee] (v7) to[out=90-15,in=270+15] node[elbl]{$e_{13}$} (v5);
    \end{tikzpicture}
    \begin{tabular}{ccc}
        Edge & Weight & Transition Matrix\\
        $e_1$ & $\rho^2$ & $\begin{pmatrix}p_3\end{pmatrix}$\\
        $e_2$ & $\rho$ & $\begin{pmatrix}p_1\end{pmatrix}$\\
        $e_3$ & $\rho$ & $\begin{pmatrix}p_2&p_1\end{pmatrix}$\\
        $e_4$ & $\rho$ & $\begin{pmatrix}p_1\end{pmatrix}$\\
        $e_5$ & $\rho$ & $\begin{pmatrix}p_2&p_1\end{pmatrix}$\\
        $e_6$ & $\rho$ & $\begin{pmatrix}0&1\\p_3&0\end{pmatrix}$\\
        $e_7$ & $\rho$ & $\begin{pmatrix}0&1&0\\p_2&0&p_1\end{pmatrix}$\\
        $e_8$ & $\rho$ & $\begin{pmatrix}0&1&0\\p_2&0&p_1\end{pmatrix}$\\
        $e_9$ & $\rho$ & $\begin{pmatrix}0&1\\p_2&p_1\\p_3&0\end{pmatrix}$\\
        $e_{10}$ & $\rho$ & $\begin{pmatrix}0&1&0\\0&p_1&0\\p_2&0&p_1\end{pmatrix}$\\
        $e_{11}$ & $\rho$ & $\begin{pmatrix}0&0&1\\0&p_1&0\\p_3&0&0\end{pmatrix}$\\
        $e_{12}$ & $\rho$ & $\begin{pmatrix}0&0&1&0\\p_2&0&0&p_1\\0&p_3&0&0\end{pmatrix}$\\
        $e_{13}$ & $\rho$ & $\begin{pmatrix}0&1&0\\0&0&1\\p_3&0&0\\p_2&0&p_1\end{pmatrix}$
    \end{tabular}
    \caption{Transition graph for the non-equicontractive finite type example}
    \label{f:ftype}
\end{figure}

\subsubsection{An example with the set of lower local dimensions not equal to the set of upper local dimensions}\label{ex:other}
The IFS with $S_{i}(x)=x/4+d_{i}/12$ for $d_{i}=i$ when $i=0,1,...,5$, $ d_{6}=8$, $d_{7}=9$, is known to be of finite type \cite{hhn2018,nw2001} and satisfies the finite neighbour condition.
The essential class is a single vertex with four outgoing edges, and there are two additional loop classes: a simple loop class with one vertex, along with a non-simple irreducible loop class with three vertices.

This example is notable since the set of lower local dimensions in the non-essential irreducible loop class need not coincide with the set of upper local dimensions (see \cref{r:ldim-exception}).

\subsubsection{A Pisot reciprocal Bernoulli convolution with a non-simple non-essential loop class}\label{ex:large-loop}
Another interesting example is the Bernoulli convolution with parameter $\rho$, where $\rho$ is the reciprocal of the Pisot root of $x^{3}-x^2-1$.
This finite type IFS has 5 maximal loop classes: the essential class with 46 elements, another irreducible loop class with 23 elements, and 3 additional simple loop classes.
For more details on this IFS, we refer the reader to \cite{hhm2015}.
\bibliographystyle{amsplain}
\bibliography{texproject/citations/local-main}
\end{document}